\documentclass[10pt, reqno]{amsart} 
\usepackage{graphicx}

\usepackage{color}

\def\mmN{\textcolor{blue}} 

\usepackage{amsfonts}
\usepackage{amsthm} 
\usepackage{amsmath} 
\usepackage{mathrsfs}
\usepackage{eucal}	

\usepackage{hyperref} 
\usepackage{varioref} 

\numberwithin{equation}{section}

\theoremstyle{plain} 
\newtheorem{theorem}{Theorem}[section]
\newtheorem{proposition}[theorem]{Proposition}

\newtheorem{corollary}[theorem]{Corollary}

\theoremstyle{definition}
\newtheorem{remark}[theorem]{Remark}
\newtheorem{example}[theorem]{Example}

\newtheorem{definition}[theorem]{Definition}

\def \R{\mathbb R}
\def \N{\mathbb N}
\def \eps{\varepsilon}
\def \U{\mathcal U}
\def \K{\mathcal K}
\def\NN{\mathcal N}
\def\SS{\mathcal S}
\def \J{\mathcal J}
\def \T{\mathcal T}
\def\V{\mathcal{V}}
\def\W{\mathcal{W}}
\def \vsm{\vskip 0.1 truecm}
\def \ds{\displaystyle}
\def\bel{\begin{equation}\label}
\def\eeq{\end{equation}}
\def \w{\omega}
\def \xb{\check x_0}
\def \weak{\rightharpoonup^*}

\begin{document}

\title[Vector-valued impulsive delay systems]{Optimal Control Problems with Vector-Valued Impulse Controls and Time Delays} 
\author{Giovanni Fusco, Monica Motta, Richard Vinter}

\keywords{Impulse Control \and Time Delay Systems \and Optimal Control \and Maximum Principle \and Bounded Variation.}
\subjclass{49N25 \and 93C43 \and 49K21 \and 26A45 \and 49J21.}
\thanks{This research is  supported by the  INdAM-GNAMPA Project 2024 ``Non-smooth Optimal Control Problems",    CUP E53C23001670001,  and by PRIN 2022 ``Optimal control problems: analysis, approximation and applications",  2022238YY5.}

%
%
%
\maketitle

\begin{abstract}
We consider a nonlinear control system with vector-valued measures as controls and with dynamics depending on time delayed states.   First, we introduce a notion of   discontinuous, bounded variation solution associated with this system and establish an equivalent representation formula for it, inspired by the approach known in delay-free impulsive control as the `graph completions' method. Then, thanks to this equivalent formulation, we prove well-posedness properties of these solutions and also derive necessary optimality conditions in the form of a Maximum Principle for an associated minimization problem. 
\end{abstract}


\section{Introduction}
 In this paper we introduce a well-posed notion of impulse control and corresponding impulse  solution for a delay impulse control system of the form (We write $\{ x(t-kh)\}_{k=0}^M$ in place of $(x(t), x(t-h), \dots, x(t-Mh))$)
\begin{equation}\label{DIS}\tag{DIS} 
\begin{cases}
\ds\frac{dx}{dt}(t)   = f(t, \{x(t-kh\}_{k=0}^M) + \sum_{j=1}^m g_j(t, \{x(t-kh\}_{k=0}^M) \frac{\,du^j}{dt}(t) ,   \\
 \ds\frac{dv}{dt}(t) = \sum_{j=1}^m\left | \frac{du^j}{dt}(t)\right |,   \\[1.0ex]
 \ds u(0)=0, \quad \frac{du}{dt}(t)\in\K, \quad \text{a.e. $t\in[0,T]$,}  \\[1.0ex]
(x,v)(0)=(\xb,0), \ \ x(t)=\xi_0 (t)  \ \  \forall t\in[-Mh,0[ ,
\end{cases}
\end{equation}
 when   absolutely continuous controls $u=(u^1,\dots,u^m):[0,T]\to\R^m$ 
  are replaced with    bounded  variation (in short, BV)  functions $u$,  so that  absolutely continuous measures $\frac{\,du}{dt}\,dt$ are replaced with   vector-valued measures $du$,  ranging over a cone $\K$. 
  Furthermore, we provide necessary optimality conditions  for associated  optimal  control problems.   Precise assumptions on data  will be stated in the next section.

Delay-free  impulse control systems of this kind,  in which $f$ and the $g_j$'s are replaced by some $\tilde f$ and $\tilde g_j$'s   not depending on the delayed states,    have been widely investigated in the literature, also in  relation to optimal impulse control problems. Without aiming to be exhaustive, let us   mention the  papers \cite{BR88,MR95,K06,SV1,AKP12,MS20},   which follow a similiar approach,  and the books \cite{MiRu,BP,AKP19},  for an overview. In all these works, a BV  control  function $u$ is considered as  an idealization of high intensity control action over a  short time interval  and   the corresponding  impulse solution is defined   in such a  way as to coincide with the limit of  some sequence of state trajectories associated with a  sequence of absolutely continuous controls converging  to $u$.  As is well-documented in the above-mentioned literature, different approximations of a BV control can give rise to different state trajectories in the limit, unless the control is scalar-valued  or the dynamics functions $\tilde g_j$ satisfy a quite strong `commutativity' hypothesis (see e.g.  \cite{BP}). In the delay-free general situation, the set of solutions associated with a given vector-valued  BV control $u$   can be parameterized by families of control functions attached to each point of discontinuity in the distribution of the measure $du$, which determine the evolution of the state during the jump.

In contrast, delay impulse systems as in \eqref{DIS}  have so far only been studied under rather restrictive assumptions on dynamics and controls. Specifically,    \cite{SF10,SZ16} introduce a notion of impulse solution and investigate its main properties, assuming that the  $g_j$'s do not depend on the delayed state variables and satisfy the commutativity condition. In  \cite{FM324}, a well-posed notion of solution and  necessary optimality conditions are introduced for an impulse control  problem in which the $g_j$'s depend on the time variable only (with possibly non-commensurate time delays in the drift term $f$).  
In these papers, given a vector-valued BV control $u$, it is possible to determine one and only one corresponding bounded variation solution, since {\em all}  sequences of trajectories corresponding to absolutely continuous controls  approximating $u$  converge to the same function.  By analogy with  delay-free   impulse control,  this uniqueness property is somehow expected, since in both cases the commutativity condition is satisfied. Consequently, it might be supposed that  extending the impulse theory  to allow for time delays in the state would be a routine exercise, in which we adapt, for example,  the necessary optimality conditions  for conventional (`non-impulse') optimal control problems with time delay established in \cite{VB,V1,V2}.  This is   refuted however by recent results in \cite{FMV1,FMV2}, where system \eqref{DIS} is studied in the case of scalar and  monotone nondecreasing controls, $M=1$, and constant $\xi_0$.  In particular,   \cite{FMV1,FMV2} show that   when the dynamics function    $g=g_1$ depends  explicitly on a state time delay,   unlike in delay-free impulse control,   different control approximations may give rise to different state trajectories in the limit, even if the   control $u$  is scalar. This  is due to the overlapping of pulses at different times, which may arise because of the time delay.  Therefore, to identify a unique impulse solution  it is not sufficient to assign $u$;  it is necessary also to specify attached controls describing the behaviour of the solution during any jump, similar to the delay-free vector impulse case. 
 
The main novelty of the class of problems addressed in this paper is the presence of time delays in the state in conjunction with   vector-valued BV controls $u$, and measures $du$  constrained to range over a closed (not necessarily convex) cone   $\K\subseteq \R^m$.  Furthermore, compared to \cite{FMV1}, we also consider  a function $\xi_0$ rather than a constant  to describe the past history of the state $x$,  as is more usual in problems with time delays.

%


Let us briefly discuss the challenges of generalizing the method used for studying impulse control systems  when we introduce time delays,  and    the additional difficulties that arise when moving from the control-scalar problem investigated  in \cite{FMV1} to the present one.  In  delay-free impulse control, a key point is  to introduce an {\em equivalent} representation  for impulse controls and  impulse solutions as defined by \cite{K06,AKP19} (see also \cite{FM224}), based on the so-called  `graph completion' technique  \cite{BR88,BP}.  According to  this approach,  given a  BV control $u$, instead of adding attached controls,  
one   ‘completes’ the graph of $u$ at the discontinuity points and  considers an  ‘arc length type’, 1-Lipschitz continuous parameterization $(\varphi_0,\varphi)$   of this graph completion.  Then, one solves a conventional  time-space control system associated with $(\varphi_0,\varphi)$.  The  composition of the spatial component of this  system with the (discontinuous) time change  given by the right inverse of $\varphi_0$  describes  an impulse state trajectory.  Such a representation of impulse solutions is crucial, because it  makes possible the derivation of  the properties of impulse trajectories and
necessary optimality conditions  via the conventional  time-space control system. 

However, the graph completion approach cannot be applied directly  in the presence of time delays, since  adopting the above procedure we  would obtain  a   non-standard delay time-space control system, in which the time delays vary with time and are control dependent. Hence, we introduce  a suitably modified notion of graph completion,  so as to associate with \eqref{DIS} a  conventional delay time-space control system,   which preserves constancy of the delays.  As regards the comparison with our recent work \cite{FMV1}, 
the crucial difficulty switching from  scalar, monotone  controls to vector-valued BV controls $u$, is that some measure component $du_j$ might range over $\R$. Hence, we may have chattering approximating sequences of $u$, with total variation $\nu$ in the limit larger than the total variation of $u$  and, in consequence, we must adopt a more delicate  concept  of attached controls that involves  both $u$  and $\nu$. This also makes the generalization to the vector case of the new reparametrization technique  developed in  \cite{FMV1}  far from straightforward. We mention that, in previous work, we  considered only the case $\xi_0$ constant, because for time-varying $\xi_0$, our formerly employed reparametization techniques gave rise to a time-space system of a non-standard type, featuring control-dependent  data.   

The main contribution of this paper is to introduce a notion of impulse control and a corresponding impulse solution for  \eqref{DIS}, for which we prove an equivalent formulation  provided through an auxiliary,  conventional  time-space control system, in the spirit of the graph completion approach. Then, using this equivalence, we show the well-posedness of this definition and
 provide necessary conditions of optimality  for an optimal control problem 
 in Mayer form  with endpoint constraints.

The paper has the following structure. In Section \ref{S_impulse} we give a precise definition of impulse processes for \eqref{DIS}  and state their main properties.  Section \ref{S1}  introduces a  conventional time-space control system. In Section \ref{Smain}  we establish a one-to-one correspondence between processes for this system and  impulse processes. In Section \ref{S_optimal} we associate with \eqref{DIS} a Mayer cost and an endpoint constraint and provide a maximum principle for the resulting  optimal impulse control problem with delays. Some conclusions and possible future research  directions are proposed in Section \ref{S_concl}. An Appendix containing the main proofs concludes the paper.



\section{Impulse Controls and Impulse Solutions}\label{S_impulse}
 Throughout the  paper we shall consider the following hypotheses.
{\em
\begin{itemize} 
\item[{\bf (H)}]
$N$, $M\in\N$, $N\geq1$, $N\ge M$,  $h>0$, and $T:=Nh$.  The control set $\K\subseteq{\mathbb R}^m$ is a closed cone, and $\xb\in {\mathbb R}^n$. The  functions  $f$,   $g_j$, $j=1,\dots, m$,  are bounded and  belong to $C^1({\mathbb R}\times({\mathbb R}^n)^{N+1}, {\mathbb R}^n)$,     $\xi_0\in  C([-Mh,0]), {\mathbb R}^n)$. 
\end{itemize}}
For any control $u\in W^{1,1}([0,T],{\mathbb R}^m)$    there exists a unique solution to \eqref{DIS}, in consequence of standard results on delay control systems  (see e.g. \cite[Th. 4.1]{VB}).  
The hypothesis that  the final time $T=Nh$ can be removed by redefining $f$ and the $g_j$'s as discontinuous functions on an extended time interval, as we show in future work.

For any  interval $[s_1,s_2]\subset{\mathbb R}$, we denote  $C^*([s_1,s_2],{\mathbb R}^m)$  the set of signed, finite and regular vector-valued measures $\mu$ from the Borel subsets of $[s_1,s_2]$ to $\R^m$, $|\mu|$  the total variation measure,\footnote{As is customary,  $|\mu|\in C^\oplus([0,T])$ and  
$
|\mu| = \sum_{j=1}^l [\mu^j_+ + \mu^j_-],
$
where $\mu^j_+, \mu^j_-\in C^\oplus([0,T])$ are the elements of the Jordan decomposition of the $j$-th component $\mu^j$ of $\mu$.} $\mu^c$ the continuous component of the measure. Given a sequence $(\mu_i)$ and $\mu$, we write   $\mu_i\weak \mu$ if  
$
\lim_i \int_{[0,T]} \psi(t) \mu^j_i(dt) = \int_{[0,T]} \psi(t) \mu^j(dt)$ for all $\psi\in  C([0,T], {\mathbb R})$ and $ j=1,\dots,m$. $C^\oplus([s_1,s_2]):=C^*([s_1,s_2],[0,+\infty[)$.
 $BV^+([s_1,s_2],{\mathbb R}^m)$ is  the subset of BV functions from $[s_1,s_2]$ to ${\mathbb R}^m$,  right continuous on  $]s_1,s_2[$ and vanishing at $s_1$. We set 
$$
\begin{array}{l}
\ds C_\K^*([s_1,s_2]):=\left\{\mu\in C^*([s_1,s_2],{\mathbb R}^m): \ \ \text{range}(\mu)\subset\K\right\}, \\[1.0ex]
 BV_\K^+([s_1,s_2]):=\left\{u\in BV^+([s_1,s_2],{\mathbb R}^m):   \  \ \ du\in C_\K^*([s_1,s_2])\right\}.
\end{array}
$$
Note that assigning a measure $\mu\in C_\K^*([s_1,s_2])$ is equivalent to fixing   $u\in BV_\K^+([s_1,s_2])$, such that $u(s_1)=0$ and $u(t)  =  \int_{[s_1,t]} \mu(ds)$ for all $t\in]s_1,s_2]$. 

Given $\mu\in C_\K^*([0,T])$, we define the subset  $\V_c(\mu)\subset C^\oplus([0,T])$ as
\bel{Vmu}
\V_c(\mu):=\left\{\nu: \ \  \nu^c=|\mu^c|,  \   \ \exists(\mu_i)\subset C_\K^*([0,T]) \ \text{  s.t. } \  (\mu_i, |\mu_i|)\weak (\mu,\nu)\right\}.
\eeq 
 In general, if $\nu\in\V_c(\mu)$    one has $\nu\ge|\mu|$ and $|\mu|$ always belongs to $\V_c(\mu)$. Actually, if the range of $\mu$ is contained in a closed convex cone which belongs to one of the orthants of ${\mathbb R}^m$, then $\V_c(\mu)=\{|\mu|\}$. 
 
 In the following, for any  $w\in{\mathbb R}^m$ we define $\|w\|:=\sum_{j=1}^m|w^j|$.

\begin{definition} \label{imp_control}
An {\em impulse control} $(\mu, \nu, \{\w^r\}_{ r \in [0,h]})$ for control system \eqref{DIS}  comprises two   Borel measures $\mu  \in C_\K^*([0,T])$,   $\nu\in\V_c(\mu)$, and a family of essentially bounded, measurable functions $\w^r =(\w^r_1,\dots, \w^r_N ) : [0,1] \rightarrow \K^N$  parameterized by $r \in [0,h]$,  with the following properties: 
\begin{itemize}
\item[(i)] For any $r \in [0,h]$, 
$
 \sum_{l=1}^N \|\w^r_l(s)\|  =  \sum_{l=1}^N \int_0^1 \|\w^r_l (s)\| ds, \ \mbox{ a.e. } s \in [0,1];
$ 
 \item[(ii.1)]  for any $r\in]0,h[$ and $l=1,\dots,N$,   $\int_0^1\|\w^r_l  (s)\|\,ds  = \nu  (\{r+(l-1)h\}) $,
 
 \item[(ii.2)]  for any   $l=0,\dots,N$,   $\int_0^1[\|\w^h_l  (s)\|+\|\w^0_{l+1}  (s)\|]\,ds  = \nu  (\{lh\}) $,
 
\item[(iii.1)]  for any $r\in]0,h[$ and $l=1,\dots,N$,  $\int_0^1 (\w^r_l)^j (s) ds =  \mu^j (\{r+(l-1)h\})$ for all $j=1,\dots, m$,
 
\item[(iii.2)]  for any   $l=0,\dots,N$,  $\int_0^1 [(\w^h_l)^j (s)+(\w^0_{l+1})^j] ds =  \mu^j (\{lh\})$, where $(\w_0^h)^j\equiv0$ and $(\w_{N+1}^0)^j\equiv0$, for all $j=1,\dots, m$.
\end{itemize}
We call a  family of functions $\{\w^r\}_{ r \in [0,h]}$  as above  {\em attached} to $(\mu,\nu)$.
An impulse control $(\mu, \nu, \{\w^r\}_{ r \in [0,h]})$ is   {\em strict sense} if $\nu$
 is absolutely continuous with respect to the Lebesgue measure.\footnote{By Def. \ref{imp_control},  the atoms of $\mu$ are always a subset of the atoms of $\nu$. However, in the vector-valued case  we can   have impulse controls in which $\nu$ has  atoms while $\mu$ does not.}  We also call strict sense control any function $u\in W^{1,1}([0,T],{\mathbb R}^m)$ with $du$ ranging on $\K$ and $u(0)=0$, which is clearly associated with the strict sense  impulse control $(du, |du|, \{\w^r\equiv 0\}_{ r \in [0,h]})$.\footnote{Since $du$ has no atoms,   by Def. \ref{imp_control},(i),  $\w^r_i= 0$ a.e.   for any $r\in[0,h]$ and  $i=1,\dots,N$.} 
\end{definition}
Def. \ref{imp_control} extends the notion of impulse control first introduced in \cite{FMV1} for delay impulse systems with nonnegative scalar  measures. It is inspired by the notion  previously introduced for delay-free  impulse systems with vector-valued controls in \cite{K06,KDPS15},  where however condition $\nu^c=|\mu^c|$ in the definition of $\V_c(\mu)$ is not present.  As  shown in \cite{FM224},  even in the absence of  delays
 this  assumption is  in fact crucial to prove the well-posedness of the  impulse process definition.  
\begin{definition}
Take an impulse control  $(\mu, \nu, \{\w^r\}_{r \in [0,h]})$. Then, $(x,v) \in  BV ([-Mh,T],{\mathbb R}^n)\times BV([0,T],{\mathbb R})$ is   an {\em impulse solution} to \eqref{DIS}  associated with $(\mu, \nu,\{\w^r\}_{ r \in [0,h]})$, and   $(\mu, \nu, \{\w^r\}_{r \in [0,h]},x,v)$  is called an {\em impulse process}, if  $x(t)=\xi_0(t)$ for $t\in[-Mh,0[$,  $(x,v)(0)= (\xb,0)$,   
$(x,v)((l-1)h)=(\zeta_l^0,\theta_l^0) (1)$ for $l=2, \dots, N+1$, and, moreover
  for each $l =1,\ldots,N$ and $t \in ](l-1)h, lh[,$\footnote{For any function $z$ defined on an interval $I$ and $t\in I$,   $z^-(t)$ and $z^+(t)$ denote the left and  right limit of $z$ at $t$, respectively.}
$$
\begin{array}{l} 
   x(t) = \zeta^0_l (1) +\int_{(l-1)h}^t  f(t', \{x(t'-kh)\}_k)dt'  \\[1.0ex]
 \hspace{0.75 in} + \sum_{j=1}^m \int_{[(l-1)h,t]}g_j(t',  \{x(t'-kh)\}_k)d(\mu^{c})^j(t')
\\ [1.0ex]
 \hspace{0.75 in}+ \underset{r \in ]0,\, t- (l-1)h] }{\sum} ( \zeta^r_l (1) - x^-(r +(l-1)h)) ,
 \\[1.5ex]
  v(t) = \theta^0_l (1) + \int_{[(l-1)h,t]} d \nu^{c}(t')+ \underset{r \in ]0,\, t- (l-1)h] }{\sum} ( \theta^r_l (1) - v^-((r +(l-1)h)),
 \end{array}
$$
where, for  any $r \in [0,h]$,      $\zeta^r_l:[0,1]\rightarrow {\mathbb R}^n$ solves  the  Cauchy problem
\begin{equation}\label{jump_x1}
 \left\{\begin{array}{l}
\ds\frac{d\zeta^r_l}{ds}  (s)  \,=\,  \sum_{j=1}^mg_j(r + (l-1)h,\{ \zeta^r_{l-k} (s)\}_k )\,(\w_l^{r})^j(s) \quad \mbox{ a.e. } s \in [0,1], \\
\zeta^r_l (0) =
\left\{
\begin{array}{ll}
x^-(r + (l-1)h) & \mbox{if } r \in ]0,h]
\\
\zeta_{l-1}^h (1)&\mbox{if } r =0 \,,
\end{array}
\right.
 \end{array}
\right.
\end{equation}
while $\theta^r_l:[0,1] \rightarrow [0,+\infty)$ is given by
\bel{jump_v1}
 \left\{\begin{array}{l} \theta^r_l  (s)  \,=\,\theta^r_l (0)+\int_0^s\|\w_l^{r}(s')\|\,ds'  \quad \mbox{ for all } s \in [0,1],\\
\text{where} \quad 
\theta^r_l (0) =
\left\{
\begin{array}{ll}
v^-(r + (l-1)h) & \mbox{if } r \in ]0,h]
\\
\theta_{l-1}^h (1)&\mbox{if } r =0 \,.
\end{array}
\right.
\end{array}
\right.
\end{equation}
In these relations, $\theta^r_0\equiv 0$, 
$
\zeta^r_0(s) :=  \left\{
\begin{array}{ll}
\xi_0(r-h) &\mbox{if  $s\in [0,1[$ or $r\in[0,h[$}
\\
\xb &\mbox{if $s=1$ and $r=h$}
\end{array}
\right.,$ and 
$\zeta^r_{-a}\equiv\xi_0(r-(a+1)h)$ for  $a=1,\dots,M-1$. 
We call an impulse process $(\mu,  \{\w^{r}\}_{ r \in [0,h]},x,v)$ corresponding to a strict sense control $u$,  written briefly as $(u,x,v)$,  a {\em  strict sense process}.\footnote{In this circumstance  $(x,v)$, which we refer to as a  strict sense solution, is the classical   $W^{1,1}$  solution to the delayed differential system \eqref{DIS} associated with $u$ (and $v$ is  the total variation function of $u$).}
 \end{definition}
 \begin{remark}\label{R_tv} Given an impulse process $(\mu,  \{\w^{r}\}_{ r \in [0,h]},x,v)$, by the very definition of $v$ it follows that $v(t)=\nu([0,t])$ for every $t\in]0,T]$.  Moreover, the impulse solution $(x,v)$ has bounded variation and is right continuous on $]0,T]$.  
 \end{remark}
%
 
We conclude this section by assembling  the main properties of impulse processes, whose proofs will be given in Section \ref{Smain}. 
\begin{proposition}\label{prop_uniq} 
Corresponding to any impulse control $(\mu, \nu, \{\w^r\}_{ r \in [0,h]})$ there exists a corresponding impulse solution $(x,v)$ to \eqref{DIS} and it is unique.
\end{proposition} 

 \begin{proposition}\label{prop_density}
Take an impulse process  $(\mu, \nu, \{\w^r\}_{r \in [0,h]},x,v)$. Then  there exists a sequence of strict sense processes $(u_i,x_i,v_i)$, such that
$$
(du_i,|du_i|=dv_i) \weak  (\mu,\nu),\quad dx_i \weak dx,  \quad (x_i,v_i)(t)\to (x,v)(t) \ \ \forall t\in E,
$$  
where $E$ is a full measure subset of $[0,T]$ which contains $0$ and $T$. 
 \end{proposition}

We conclude with a  compactness result for the subsets of impulse trajectories corresponding to impulse controls  with equibounded total variation.
\begin{proposition} \label{robust_th}
Let  $(\mu_i, \nu_i, \{\w_i^r\}_{r \in [0,h]}, x_i, v_i)$ be   a sequence of impulse processes such that the sequence $(\nu_i)$ is uniformly bounded in total variation by some $C>0$.  If  the cone $\K$ is convex,  then     there exist  an impulse process $(\mu, \nu, \{\w^r\}_{r \in [0,h]},x,v)$ and a full measure subset  $E\subseteq [0,T]$ with  $0,T\in E$,  such that, along a subsequence, 
$$
\mu_i\weak \mu, \quad x_i (t) \rightarrow x(t) \ \ \forall t\in E, \quad v(T)\leq C.
$$ 
 Moreover, if the stronger condition that the set  $\{(w,\|w\|): \  w\in \K\}$  is convex holds, then also $v_i(t)\to v(t)$ for all $t\in E$.
 \end{proposition}
 
 Note that when $\K$ is   one of the orthants,    $\{(w,\|w\|): \  w\in \K\}$  is convex.

%

%
 
 \section{Graph Completions and Graph Completion Solutions}\label{S1}
As discussed in the Introduction,  the graph completion technique developed for delay-free impulse control systems  (see \cite{BR88,MiRu}) cannot  be directly applied in the presence of time delays. To illustrate our  approach,  let us  first consider a strict sense  control $u\in W^{1,1}([0,T])$ with $u(0)=0$ and $du\in\K$.  Inspired by the   `method of steps'  for conventional delay systems (see e.g. \cite{Driver}) we define  the $N$-uple  $(u_1,u_2,\dots,u_{N})(t):=(u(t), u(t+h), \dots,u(t+(N-1)h)$,    $t\in[0,h]$,  and choose a Lipschitz continuous parameterization $[0,S]\ni s\mapsto(\varphi_0,\varphi)(s)= (\varphi_0 ,(u_1,u_2,\dots,u_{N})\circ\varphi_0)(s)$  of the graph of $(u_1,u_2,\dots,u_{N})$ on $[0,h]$, with time-change $\varphi_0$  strictly increasing.
In this way, we arrive at the following auxiliary (conventional and delay-free) time-space control system 
\bel{aux_system} \tag{ACS}
\begin{cases}
\ds\frac{d\tau_l}{ds}(s) = \frac{d\varphi_0}{ds}(s), \\
\ds\frac{dy_l}{ds}(s)= f\left(\tau_l(s), \{y_{l-k}(s)\}_{k}\right) \frac{d\varphi_0}{ds}(s) + \sum_{j=1}^m g_j\left(\tau_l(s), \{y_{l-k}(s)\}_{k}\right) \frac{d\varphi_l^j}{ds}(s), \\
\ds\frac{d\beta_l}{ds}(s)= \left\|\frac{d\varphi_l}{ds}(s)\right\|, \qquad \text{a.e. $s\in[0,S]$,\ \  $l=1,\dots,N$}, \\
 y_{-a}(s)=\xi_0(\tau_1(s)-(a+1)h) \quad \forall s\in[0,S],  \quad a=0,\dots,M-1, \\ 
 (\tau_1,y_1,\beta_1)(0)=(0,\xb,0), \\ (\tau_{l},y_{l},\beta_l)(0)=(\tau_{l-1},y_{l-1},\beta_{l-1})(S),  \quad l=2,\dots,N.
\end{cases}
\eeq
Let $(\tau,y,\beta)(s)=(\tau_1,\dots,\tau_N,y_1,\dots,y_N,\beta_1,\dots,\beta_N)$  be the solution to \eqref{aux_system} associated with $(\varphi_0,\varphi)$. Define 
 $(\tilde \tau,\tilde \varphi, \tilde y,\tilde\beta)$ setting, for any $z\in\{\varphi, \tau,y,\beta\}$,  
 \bel{exsol}
\ds\tilde z(s):=\sum_{l=1}^N z_l(s-(l-1)S)\chi_{_{[(l-1)S,lS[}}(s), \quad
\ds\tilde S:=NS,  \quad \tilde z(\tilde S)=z_N(S).
\eeq
Then, if   $\tilde\sigma$ is the inverse  of $\tilde \tau$,   $(x,v):=(\tilde y,\tilde\beta)\circ\tilde\sigma$ 
is the solution of  \eqref{DIS} associated with the control  $u$, which in turn coincides with $\tilde \varphi\circ\tilde\sigma$.   Observe that  the time-change  $\tilde\sigma$  is  uniform  on each interval $[(l-1)h,lh]$,  because $\tau_l(s)=(l-1)h+\varphi_0(s)$ for any   $l=1,\dots, N$, and,   at each instant $t$,  depends  on the behaviour of the control $u$  on the whole interval $[0,T]$,  not only   up to $t$.\footnote{ This fact is clarified later, when we observe that we can limit ourselves to consider canonical reparameterizations $(\varphi_0,\varphi)$, which satisfy \eqref{lip1}, so that $\frac{d\varphi_0}{ds}=1-\sum_{l=1}^N\left\|\frac{d\varphi_l}{ds}\right\|$.  }

When $u\in BV$, it is still possible to construct a  Lipschitz continuous parameterization $(\varphi_0,\varphi)$ of the graph of $(u_1,u_2,\dots,u_{N})$ with $\varphi_0$ increasing (but possibly not strictly increasing), solve the corresponding control problem \eqref{aux_system} and use the solution $(\tau,y,\beta)$ to construct a generalized solution of the delayed control system \eqref{DIS}. With this purpose we introduce the following definitions.

 \begin{definition} \label{extended controls} 
For any  $L>0$, $S>0$,  let $\U^L_\K([0,S])$ be the  subset of $L$-Lipschitz continuous  functions    $(\varphi_0,\varphi)=(\varphi_0, \varphi_1,\dots,\varphi_N):[0,S]\to [0,h]\times({\mathbb R}^m)^{N}$ such that 

\begin{itemize}
\item[(i)] if $0\leq r < s \leq S$, then $\varphi_0(r)\leq \varphi_0(s)$; \ $(\varphi_0(0),\varphi_0(S))=(0,h)$;
\vsm
\item[(ii)]  $\varphi_1(0)=0$, \  $\varphi_{l }(0)=\varphi_{l-1}(S)$ for all $l=2,\dots, N$;
\vsm
 \item[(iii)] $\ds\frac{d\varphi_l}{ds}(s)\in\K$  a.e. $s\in[0,S]$, for all $l=1,\dots, N$.
\end{itemize}
We set $\U_\K([0,S]):=\cup_{L>0}\U^L_\K([0,S])$ and call {\em extended controls} 
the elements $(S,\varphi_0,\varphi)$  of  $\ds \cup_{S>0}\{S\}\times \U_\K([0,S])$.  
 \end{definition}
 
  \begin{definition} \label{extended processes} 
Given an extended control  $(S,\varphi_0,\varphi)$, the solution $(\tau,y,\beta)\in W^{1,1}([0,S],{\mathbb R}_{\ge0}^N\times({\mathbb R}^n)^N\times{\mathbb R}_{\ge0}^N)$ of \eqref{aux_system} is said to be an {\em extended solution}  corresponding to $(S,\varphi_0,\varphi)$. We refer to  $(S,   \varphi_0,\varphi,\tau,y,\beta)$ as an  {\em extended process}.
 \end{definition}

 \begin{definition}\label{graphcompl_def}
Given $u\in BV^+_\K([0,T])$, 
a {\em graph completion (in short, g.c.)} of $u$ is any pair $(\varphi_0,\varphi)=(\varphi_0, \varphi_1,\dots,\varphi_N)\in  \U_\K([0,S])$ for some $S>0$,  satisfying  (i) $(T,u(T))=(\tilde\tau,\tilde\varphi)(\tilde S)$; (ii) for all $t\in[0,T]$, there exists $s\in[0,\tilde S]$ satisfying $(t,u(t))=(\tilde\tau,\tilde\varphi)(s)$,  for $(\tilde \tau,\tilde \varphi)$ and  $\tilde S$ as in \eqref{exsol}
\end{definition}

 
In the sequel, given a (possibly not strictly) increasing function  $A:[T_1,T_2]\to[S_1,S_2]$, right continuous on $]T_1,T_2]$ and such that $A(T_1)=S_1$, $A(T_2)=S_2$,   the {\em right inverse} of $A$  is the function $B:[S_1,S_2]\to[T_1,T_2]$ such that
$$B(S_1)=T_1, \ \
 B(s)=\inf\{t\in[T_1,T_2]: \ \ A(t)>s\} \ \text{for $s\in]S_1,S_2[$,} \ \  B(S_2)=T_2.
 $$

Given a control $u\in BV^+([0,T],{\mathbb R}^m)$, there  always exists
a natural graph completion,  obtained by ``bridging"  the discontinuities of $u$  on the graph  by straight  lines, as in the following  example. 
 
\begin{example} \label{Ex rectilinear gc}
Given   $u\in BV^+([0,T],{\mathbb R}^m)$,  
  let  $\mu\in C^*([0,T],{\mathbb R}^m)$ denote the measure such that   
$u(t) = \int_{[0,t]}   \mu(ds)$ for all $ t\in]0,T]$,  $u(0)=0$,
and consider the function $V:[0,h]\to{\mathbb R}_{\ge0}$, given by
\bel{V}
 V(r):=  \int_{[0,r]} |\mu|(dr')+\sum_{l=2}^{N } \int_{](l-1)h,(l-1)h +r]} |\mu|(dr')  \quad \forall r\in]0,h], \ \ V(0)=0. 
\eeq
Note that, if we set $u_1(r):=u(r)$, $u_2(r):=u(r+h),\dots,u_N(r):=u(r+(N-1)h)$ for any $r\in[0,h]$,\footnote{We have $u_l(0)=u_{l-1}^+(h)=u_{l-1}(h)$ for all $l=2,\dots, N$, because $u$ is right continuous.} then $V$ is sum of the total variations of the $u_l$'s.
 Define   the function $r\mapsto \psi(r):= r+V(r)$ for any  $r\in[0,h]$ and set $S:= h+V(h)$.  Since  $\psi:[0,h]\to[0,S]$  is right continuous on $]0,S]$ and   strictly  increasing,  its right inverse  $\varphi_0:[0,S]\to[0,h]$ is continuous, increasing, $\varphi_0(0)=0$,  $\varphi_0(S)=h$, and $\varphi_0$ is constant exactly on the countable number of nondegenerate intervals $[\psi^-(r),\psi^+(r)]$\footnote{We set $\psi^-(0):=0$ and $\psi^+(h):=\psi(h)=S$. Note that $\psi^+(r)=\psi(r)$ for all $r\in]0,h]$.}   corresponding to the discontinuity points of $V$.  
 
 For any  $r\in]0,h]$   at which  $V$ is discontinuous, let $I_r$ denote the (nonempty) subset of $l\in\{1,\dots,N\}$ such that $u$ has a jump  at
   $t=r+ (l-1)h$. Then, 
\bel{variation}
\sum_{l\in I_r}\|u(r+(l-1)h  ) - u^-(r+(l-1)h )\| =V(r)-V^-(r)=\psi^+(r)-\psi^-(r).
\eeq
If $r=0$,    set $V^-(0):=0$ and $u^-(0):=0$, so that 
$
\|u^+(0)-u^{-}(0)\|= V^+(0)-V^-(0)=\psi^+(0)-\psi^-(0).
$
We can now define a graph completion $(\varphi_0,\varphi)=(\varphi_0, \varphi_1,\dots,\varphi_N)$ of $u$ as follows. We take $S$ and $\varphi_0$ as  above and,
  for every $l=1,\dots,N$ and  $s\in\mmN{[}0,S]$, we define $\varphi_l(s)=u(\varphi_0(s)+(l-1)h)$ if $u$ is continuous at $\varphi_0(s)+(l-1)h$; otherwise,  we set $r:=\varphi_0(s)$ and define
\bel{varphil}
\varphi_l(s)=\frac{s-\psi^-(r)}{\psi^+(r)-\psi^-(r)} u^+(r+(l-1)h) + \frac{\psi^+(r)-s}{\psi^+(r)-\psi^-(r)} u^-(r+(l-1)h)
\eeq
for all $s\in [\psi^-(r),\psi^+(r)]$. It is easy to check that such $(\varphi_0,\varphi)$ meets conditions (i)--(ii) of Def. \ref{graphcompl_def}. Moreover, it     satisfies
\bel{lip1}
\frac{d\varphi_0}{ds}(s) + \sum_{l=1}^N \Big\|\frac{d\varphi_l}{ds}(s) \Big\|  =1 \qquad \text{for a.e. $s\in[0,S]$.}
\eeq
Indeed, this relation  follows from the very definitions of $(\varphi_0,\varphi)$ and $V$ when $u$ is continuous at $\varphi_0(s)+(l-1)h$ for any $l=1,\dots,N$. Otherwise, it follows from \eqref{variation}, \eqref{varphil}, since $\frac{d\varphi_0}{ds}(s)=0$ for a.e. $s\in [\psi^-(r),\psi^+(r)]$. 
\end{example}
 
\begin{definition} \label{gcsol}
Given a control  $u\in BV^+_\K ([0,T])$,
let $(\varphi_0,\varphi)$ be a  g.c.  of $u$ defined on some interval $[0,S]$,    $S>0$, and let $(\tau,  y,\beta)$ be the corresponding $W^{1,1}$ solution of \eqref{aux_system}. Let $(\tilde S, \tilde \tau,\tilde\varphi,\tilde y,\tilde\beta)$ be as in \eqref{exsol}.   We call  the {\em graph completion}, in short {\em g.c.,  solution}   of \eqref{DIS} associated with    $(\varphi_0,\varphi)$,  the BV pair 
\bel{sol_gen}
(x_{(\varphi_0,\varphi)},v_{(\varphi_0,\varphi)})(t):=(\tilde y,\tilde\beta)(\tilde\sigma(t)) \ \  \text{for $t\in[0,T]$,} 
\eeq
in which $\tilde\sigma:[0,T]\to [0,\tilde S]$ is the right inverse of $\tilde\tau$. 
\end{definition}
System \eqref{aux_system} is {\em rate-independent}, namely
given any strictly increasing, surjective,  Lipschitzean   function $\delta:[0, S]\to [0, \hat S]$ with Lipschitz continuous inverse, then $(\hat S, \hat\varphi_0, \hat\varphi, \hat\tau,\hat y,\hat \beta)$ is an extended  process  if  and only if the process $(S,\varphi_0,\varphi, \tau,y,\beta)$ such that
 $(\varphi_0,\varphi, \tau,y,\beta)(s):=\left(\hat \varphi_0, \hat\varphi, \hat\tau,\hat y,\hat \beta\right)\circ\delta(s)$,  for any $s\in[0,S]$,
is an extended   process.
Owing to this property, as we illustrate in the sequel, we can limit ourselves to considering {\em canonical}  extended controls and {\em canonical} extended processes, defined as extended controls $(S,\varphi_0,\varphi)$ possessing property \eqref{lip1} together with their corresponding extended processes. 

\begin{definition} \label{def_canonical} Given   $(S,\varphi_0,\varphi, \tau,y,\beta)$, we  consider the total variation function of $(\varphi_0,\varphi)$, i.e. 
$[0,S]\ni s\mapsto  \delta_c (s):=\int_0^s\Big(\frac{d\varphi_0}{ds}(s')+\sum_{l=1}^N\Big\|\frac{d\varphi_l}{ds}(s')\Big\|\Big)\,ds',
$
  and define  the  {\em canonical parameterization} of  $(S,\varphi_0,\varphi, \tau,y,\beta)$,  as   the extended process
 $$
(S^c,\varphi_0^c,\varphi^c, \tau^c,y^c,\beta^c):= (\delta_c(S),(\varphi_0,\varphi, \tau,y,\beta)\circ \delta_c^{-1}),
 $$ where $\delta_c^{-1}$ is the right inverse of $\delta_c$.\footnote{The function $\delta_c$ might not be strictly increasing, but in view of the   form of the dynamics,  the canonical parametrization is still an extended process with Lipschitz constant 1.} Hence,  we  call {\it equivalent},  and write
$
(\hat S, \hat\varphi_0,\hat\varphi, \hat\tau,\hat y,\hat \beta)  \sim (S,\varphi_0,\varphi, \tau,y,\beta),$
 any two extended   processes with the same canonical parameterization. Clearly, $(\hat\varphi_0,\hat\varphi, \hat\tau,\hat y,\hat \beta)$ and $(\varphi_0,\varphi, \tau,y,\beta)$ are different parameterization of the same curve. 
 \end{definition}
It follows immediately that the canonical parametrization of an extended process is a canonical process and   an extended  process is canonical if and only if it  coincides with its canonical parameterization.    As a corollary, we get the following relationships between the associated g.c. solutions  (see \cite[Sect. 3]{MR95}).
  \begin{proposition} \label{equivalent_graph}
Let   $(S,\varphi_0,\varphi)$ be an extended control.  Then, the   g.c. solution $(x_{(\varphi_0,\varphi)},v_{(\varphi_0,\varphi)})$ of \eqref{DIS}   coincides with   $(x_{(\varphi^c_0,\varphi^c)},v_{(\varphi^c_0,\varphi^c)})$. In particular, setting
$\tilde\delta_c(s)=\sum_{l=1}^N[\delta_c(s-(l-1)S)+(l-1)S^c]\chi_{_{[(l-1)S,lS[}}$  for $s\in[0,\tilde S[$, $\tilde\delta_c(\tilde S)= NS^c$,
we have that $(x_{(\varphi_0,\varphi)},v_{(\varphi_0,\varphi)}) =(\tilde y^c, \tilde\beta^c)\circ(\tilde\delta_c\circ\tilde\sigma)$.\footnote{Here, $\tilde\sigma$  is as in Def. \ref{gcsol} and $\tilde y^c$,  $\tilde\beta^c$ are as in \eqref{exsol},  with $(S^c,\varphi_0^c,\varphi^c, \tau^c,y^c,\beta^c)$  in place of $(S,\varphi_0,\varphi, \tau,y,\beta)$.}  
\end{proposition}


\begin{remark}\label{Rconv} The introduction of all Lipschitz continuous extended controls, not only the canonical ones, plays a key role with regards to minimization problems over extended processes (and, consequently, over impulse processes, in view of the equivalence result in Thm. \ref{equivalence_th} below). Indeed, on the one hand,  considering all extended controls with a Lipschitz constant between 0 and 1   guarantees a convexity property that implies a compactness result for extended solutions, and hence  existence of a minimizer (see Props. \ref{robust_th}, \ref{P_existence}). In addition, it is helpful in deriving  a Maximum Principle, as  from Prop. \ref{equivalent_graph} it follows that: 1) a minimizer  can always be assumed canonical; 2) a canonical minimizer is minimal even among extended processes $(S,\varphi_0,\varphi, \tau,y,\beta)$ with its same end-time and satisfying e.g.
$\frac{1}{2}\le\frac{d\varphi_0}{ds}(s) + \sum_{l=1}^N \Big\|\frac{d\varphi_l}{ds}(s) \Big\|  \le\frac{3}{2}$   a.e.   
(so that we can replace a free end-time problem by a fixed end-time one). 
\end{remark}

\section{Equivalence between  Impulse and Extended Processes}\label{Smain}
We start  by introducing some notation.  Given  an  impulse control $(\mu,\nu,\{\w^r\}_{r\in[0,h]})$, for any $l=1,\dots, N$,  let $\mu_l\in C_\K^*([0,h])$ and  $\nu_l\in C^{\oplus}(0,h)$ be the Borel measures with distribution functions 
 \bel{nul}
 \begin{array}{l}
 \mu_l([0,r]):=\begin{cases} \ds \int_0^1\w_l^0(s)\,ds+\mu(](l-1)h,r+(l-1)h]) \quad\text{if $r\in[0,h[$,} \\
\ds \int_0^1\w_l^0(s)\,ds+\mu(](l-1)h,h[)+\int_0^1\w_l^h(s)\,ds \quad\text{if $r=h$,}
 \end{cases} \\
\nu_l([0,r]):=\begin{cases} \ds \int_0^1\|\w_l^0(s)\|\,ds+\nu(](l-1)h,r+(l-1)h]) \quad\text{if $r\in[0,h[$,} \\
\ds \int_0^1\|\w_l^0(s)\|\,ds+\nu(](l-1)h,h[)+\int_0^1\|\w_l^h(s)\|\,ds \quad\text{if $r=h$,}
 \end{cases}
 \end{array}
\eeq
respectively. Notice that   $\mu_l(\{r\})=\mu(\{r+(l-1)h\})$ and  $\nu_l(\{r\})=\nu(\{r+(l-1)h\})$ if $r\in]0,h[$, while $\mu_l(\{r\})=\int_0^1\w_l^r(s)\,ds$ and $\nu_l(\{r\})=\int_0^1\|\w_l^r(s)\|\,ds$ if $r\in\{0,h\}$,  so that  Def. \ref{imp_control} implies the relations
\bel{mu_i_rel}
\begin{array}{l}
\int_0^1\w_l^r(s)\,ds=\mu_l(\{r\}), \quad \int_0^1\|\w_l^r(s)\|\,ds=\nu_l(\{r\}), \quad\text{for $r\in[0,h]$, } \\[1.5ex]
  \mu_l(\{h\})+\mu_{l+1}(\{0\})=\mu(\{lh\}), \   \nu_l(\{h\})+\nu_{l+1}(\{0\})=\nu(\{lh\}), 
  \end{array}
\eeq
for  $l=1,\dots,N$ and $l=1,\dots,N-1$, respectively. Define  $\phi:[0,h]\to[0,S]$,  as
\bel{phi}
\phi(0):=0, \quad \phi(r):= r + \sum_{l=1}^N \nu_l([0,r]) \quad\mbox{for   }r \in ]0,h ]\,, \quad S:= \phi(h).
\eeq
Let $\J\subset[0,h]$ denote the set  of  the discontinuity points of $\phi$, which is at most countable  and,  for any $r\in[0,h]$,   set  
$\Sigma^r:=[\phi^-(r), \phi^+(r)]$ (where $\phi^-(0)=0$,  $\phi^+(h)=\phi(h)$).
 Since the map $ \phi $ is right continuous on $]0,h]$, strictly increasing,  and $\phi(t_2)-\phi(t_1)\ge t_2-t_1$ for all $0\le t_1<t_2\le h$,  then its right inverse  is $1$-Lipschitz continuous,  increasing and  constant exactly on the  intervals $\Sigma^r$ with  $r\in\J$. 
 Let $l=1,\dots,N$.  Observe that   each $r\in [0,h]\setminus\J$ is not an atom for $\mu_l$ and $\nu_l$,  and $r+(l-1)h$ is not an atom for the original measures $\mu$ and $\nu$. 
Moreover, the Lebesgue measure $\mu_L$ and   the continuous components  $\mu_l^{c}$  and $\nu_l^{c}$ 
of  $\mu_l$ and $\nu_l$   
are absolutely continuous w.r.t. the measure $d\phi$. Let $m_0$, $m_l$ and $m_l^\nu$ be the Radon-Nicodym derivatives w.r.t. $d\phi$ of $\mu_L$,  $\mu_l^{c}$,  and $\nu_l^c$, respectively. Note that, assuming  $\nu^c=|\mu^c|$  in the definition of $\V_c(\mu)$, we have  $m_l^\nu=\|m_l\|$, by well-known properties of Radon-Nicodym derivatives. Hence,  $0\leq m_0(r), \   \|m_l(r)\| \leq 1$  $d\phi$-a.e,  $m_0(r)=0$ and  $m_l(r) =0$ for all $r\in\J$,  and
\begin{equation}
\label{normalize}
m_0(r) +\sum_{l=1}^N \,  \|m_l(r)\| =1 , \quad\mbox{$d\phi$-a.e. } r \in [0,h]\backslash \J.
\end{equation}
Finally, for every $r\in\J$  define
   \bel{gamma_rho}
 \gamma^r(s):=\frac{s-\phi^-(r)}{\phi^+(r)-\phi^-(r)},\quad \text{for any } s \in \Sigma^r.
 \eeq 

\begin{theorem}\label{equivalence_th}
{\rm (i)} Take  an  impulse process $(\mu,\nu,\{\w^r\}_{r\in[0,h]},x,v)$. Let  $\phi$, $S$, $\J$,   $\Sigma^r$ and the $m_l$'s be as above and define  $(S,  \varphi_0, \varphi, \tau, y, \beta)$ as follows. 
 The function $\varphi_0:[0,S]\to[0,h]$ is the right  inverse of $\phi$. Let  $l=1,\dots,N$.   The map   $\varphi_l:[0,S]\to{\mathbb R}^m$  solves
\bel{phi_l}
\frac{d\varphi_l}{ds}(s)=
\begin{cases}
m_l(\varphi_0(s)) \qquad & \text{if $s\in [0,S]\setminus \bigcup_{r\in\J} \Sigma^r$}\\
\frac{1}{\phi^+(r)-\phi^-(r) }\,\w^r_l (\gamma^r(s)) \qquad &\text{if $s\in\Sigma^r$ for some $r\in\J$,}\\
\end{cases}
\eeq
 with initial condition $\varphi_{l }(0)=0$ if $l=1$, $\varphi_{l }(0)=\varphi_{l-1}(S)$ for  $l=2,\dots, N$;  $\tau_l(s):=(l-1)h+\varphi_0(s)$ for any $s\in[0,S]$. We set $y_{-a}(s):= \xi_0(\tau_1(s)-(a+1)h)$ for $s\in[0,S]$ and $a=0,\dots,M-1$, and define
\bel{cond_finy}
(y_l,\beta_l)(0):= \begin{cases} (\xb,0) &\quad\text{if $l=1$,}\\
(\zeta^h_{l-1},\theta^h_{l-1}) (1) &\quad\text{if $l>1$,}
\end{cases} \quad (y_l,\beta_l)(S):= (\zeta^h_l , \theta^h_l)(1),
\eeq
and
\bel{equiv_tr2}
 (y_l,\beta_l)(s)=\left\{\begin{array}{l}
 (x,v)\circ \tau_l(s), \qquad s\in]0, S[\setminus \cup_{r\in\J}  \Sigma^r , \\[1.0ex]
\ds \big(\zeta^r_l\big(\gamma^r(s)\big), \theta^r_l\big(\gamma^r(s)\big)\big) \ \  s\in ]0, S[\cap\Sigma^r, \ \   r\in\J,
  \end{array}\right.
\eeq
for   $\zeta^r_l$,  $\theta^r_l$  as in \eqref{jump_x1}, \eqref{jump_v1}, respectively.
Then, the element  $(S,  \varphi_0, \varphi, \tau, y, \beta)$ is a canonical extended process and 
it satisfies  the relation
\bel{equiv_tr}
 ( x,v) (t)=(\tilde y,\tilde\beta)\circ\tilde\sigma(t), \quad t\in[0,T], 
 \eeq
 in which $\tilde\sigma$ is the right inverse of $\tilde\tau$, $\tilde\tau$  as in \eqref{exsol}. In other words, $(x,v)$ coincides with the g.c. solution associated with $(S,  \varphi_0, \varphi, \tau, y, \beta)$.

\vsm
\noindent
{\rm (ii)} Conversely, take a canonical extended process  $(S,  \varphi_0, \varphi, \tau, y, \beta)$ and 
define  $(\mu,\nu,\{\w^r\}_{r\in[0,h]},x,v)$ as follows. Let $\sigma$ be the right inverse of $\varphi_0$ and let  $\tilde\sigma:[0,T]\to [0,\tilde S]$ be  the right inverse of $\tilde\tau$,  $(\tilde S,\tilde\tau,\tilde\varphi,\tilde y,\tilde\beta)$  as in \eqref{exsol}.
The  measures $\mu\in C_\K^*([0,T])$,   $\nu\in C^{\oplus}(0,T)$
are defined via their distribution functions, by 
\bel{u_v}
\ds \mu([0,t]): =  \int_{0}^{\tilde\sigma^+(t)} \frac{d\tilde\varphi}{ds}(s) ds, \quad  \nu([0,t]):=\int_{0}^{\tilde\sigma^+(t)}\left\| \frac{d\tilde\varphi}{ds}(s) \right\|ds,  \quad \text{for $t\in[0,T]$.} 
\eeq
For any $r\in[0,h]$  and  $l=1,\dots,N$,  set
\[
\w^r_l (s) :=(\sigma^+(r)-\sigma^-(r)) \frac{d\varphi_l}{ds} \Big((\sigma^+(r)-\sigma^-(r))\,s + \sigma^-(r)\Big), \quad\text{a.e.   $s\in[0,1]$.}
\]
Finally, 
let $(x,v)$ be  defined as in \eqref{equiv_tr}, namely  the g.c. solution associated with $(S,  \varphi_0, \varphi, \tau, y, \beta)$.
Then, $(\mu,\nu,\{\w^r\}_{r\in[0,h]},x,v)$ is an impulse process. Moreover,  the  canonical extended process corresponding to such $(\mu,\nu,\{\w^r\}_{r\in[0,h]},x,v)$ constructed as in statement {\rm(i)} is precisely $(S,  \varphi_0, \varphi, \tau, y, \beta)$.
\end{theorem}
 
 From this theorem, proved in  Appendix \ref{proof_mainth}, it follows directly:
\begin{corollary} \label{invert_cor}
 Let  ${\mathcal I}$ be the map which associates with any impulse process \linebreak $(\mu,\nu,  \{\w^{r}\}_{ r \in [0,h]},x,v)$   the canonical extended process $(S,  \varphi_0, \varphi, \tau, y, \beta)$ as  in Thm. \ref{equivalence_th}, (i). Then, ${\mathcal I}$ is invertible and  
$(\mu,\nu,  \{\w^{r}\}_{ r \in [0,h]},x,v)$ 
coincides with the impulse process associated with $(S,  \varphi_0, \varphi, \tau, y, \beta)$ as in  Thm. \ref{equivalence_th}, (ii). \end{corollary}
 
We can now prove the properties of impulse solutions stated in Section \ref{S_impulse}. 

\begin{proof}[of Prop. \ref{prop_uniq}] Given  an impulse control $(\mu, \nu, \{\w^r\}_{r \in [0,h]})$, let 
 $(S,\varphi_0,\varphi)$  be the corresponding  extended control as in Thm. \ref{equivalence_th},(i), with which a unique $W^{1,1}$  solution $(\tau, y,\beta)$ of  system of \eqref{aux_system} is associated,  by standard results on ODEs. Then, by Thm. \ref{equivalence_th},(ii), there exists at least one impulse solution $(x,v)$ corresponding to $(\mu, \nu, \{\w^r\}_{r \in [0,h]})$, which coincides with the g.c. solution associated with $(S,\varphi_0,\varphi, \tau, y,\beta)$.  Suppose by contradiction that   $(\hat x,\hat v)$ is another impulse solution associated with $(\mu, \nu, \{\w^r\}_{r \in [0,h]})$. But then, from Thm. \ref{equivalence_th}  it follows that both $(\hat x,\hat v)$ and $(x,v)$ are g.c. solutions corresponding to $(S,\varphi_0,\varphi,\tau,y,\beta)$ and satisfy \eqref{equiv_tr}, hence they coincide. \qed
\end{proof}

\begin{proof}[of Prop. \ref{prop_density}] Take an impulse process  $(\mu,\nu,\{\w^r\}_{r\in[0,h]},x,v)$. Consider the associated canonical extended process $(S,  \varphi_0, \varphi, \tau, y, \beta)$ as defined in Thm.  \ref{equivalence_th},(i), so that \eqref{equiv_tr} holds. By \cite[Thm. 5.1]{AR15}  and \cite[Thm. 5.2]{MS218},   the right inverse $\sigma$ of $\varphi_0$ can be point-wisely approximated by  a sequence of   strictly increasing, absolutely continuous, onto  functions $\sigma_i:[0,h]\to[0,S]$  with  strictly increasing, 1-Lipschitzean inverse functions   $\varphi_{0_i}$  which converge uniformly to $\varphi_0$. For each $i\in\N$, let $(S,  \varphi_{0_i}, \varphi, \tau_i, y_i, \beta_i)$ be the  extended process associated with the extended control  $(S,  \varphi_{0_i}, \varphi)$, so that $(\tau_i,y_i,\beta_i)$ converges uniformly to $(\tau,y,\beta)$ by \cite[Thm. 1]{BR88}. Let $(\tilde\varphi, \tilde\tau_i,  \tilde y_i, \tilde \beta_i)$ be defined according to \eqref{exsol} and set
$(u_i,x_i,v_i):=(\tilde\varphi, \tilde y_i, \tilde \beta_i)\circ\tilde\sigma_i,
$ 
where $\tilde\sigma_i$ is the inverse of $\tilde\tau_i$. Then, by Thm.  \ref{equivalence_th},(ii), for each $i\in\N$, $(u_i,x_i,v_i)$ 
  is a strict sense process and it satisfies\footnote{In general, $(S,  \varphi_{0_i}, \varphi, \tau_i, y_i, \beta_i)$ is not canonical, but   the associated g.c. solution coincides with  the one corresponding to its canonical parameterization, by Prop. \ref{equivalent_graph}.}
$\underset{i\to+\infty}{\lim}u_i(t) =\int_0^{\tilde\sigma(t)}\frac{d\tilde\varphi}{ds}(s)ds=\mu([0,t])$ for $t\in[0,T]\setminus\{0,h,\dots,(N-1)h\}$, and  
$\underset{i\to+\infty}{\lim}(x_i,v_i)(t) =(x,v)(t)$ for $t\in[0,T]\setminus\{h,\dots,(N-1)h\}$.
Then,  $(du_i,|du_i|=dv_i) \weak  (\mu,\nu)$ and $dx_i \weak dx$ by  \cite[Prop. 5.2]{VP}. \qed
\end{proof}

\begin{proof}[of Prop. \ref{robust_th}] 
Take a sequence $(\mu_i, \nu_i, \{\w_i^r\}_{ r \in [0,h]}, x_i, v_i)_{i\in\N}$ of impulse processes such that   $\sup_{i\in\N}v_i(T)\le C$ for     some $C>0$ (see Rem. \ref{R_tv}). For each $i\in\N$, let $(S_i,\varphi_{0_i},\varphi_i,\tau_i,y_i,\beta_i)$ be the  canonical extended process associated with  $(\mu_i, \nu_i, \{\w_i^r\}_{ r \in [0,h]}, x_i, v_i)$  as in Thm. \ref{equivalence_th},(i) and define $(\tilde\varphi_i, \tilde\tau_i,\tilde y_i,\tilde\beta_i)$ as in \eqref{exsol}.   Since 
$
S_i=\varphi_{0_i}(S_i)+
\beta_{N_i}(S_i)=h+v_i(T)\le h+C,
$ along a subsequence, one has $S_i\to S$, for some $S >0$. Moreover, setting $(\w_{0_i},\w_i):=\left(\frac{d\varphi_{0_i}}{ds}, \frac{d\varphi_{_i}}{ds}\right)$,  $(\varphi_i,\tau_i,y_i,\beta_i)$ is a solution  on $[0,S_i]$ associated with  $(\w_{0_i},\w_i)$,
for   the space-time control system obtained from \eqref{aux_system}, when instead of $(S,\varphi_0,\varphi)$ we consider  as control  $(\w_0,\w):=\left(\frac{d\varphi_{0 }}{ds}, \frac{d\varphi}{ds}\right)$, taking values in the  compact set
$
W:=\left\{ (w_0,w)\in[0,+\infty[\times \K^N \text{ : } w_0+\sum_{l=1}^N \|w_l\|\leq 1 \right\},
$
and  add the equations 
\bel{eq_phi}
 \ds\frac{d\varphi}{ds}(s)= \w(s) \  \text{a.e. $s\in[0,S]$,} \ \ \varphi_1(0)=0, \ \ \varphi_l(0)=\varphi_{l-1}(S), \ l= 2,\dots,N.
\eeq 
 Suppose first that the cone $\K$  is convex. Note that the system under consideration is not control-affine, unless any equation for  the $\beta_l$ components that involves $\|\w_l\|$  is excluded.  In particular,  the convexity of the control set $W$ does not imply that of  the whole sets of velocities. Nevertheless, if we disregard the $\beta_l$'s components,   we can invoke the compactness of trajectories theorem \cite[Thm. 2.5.3]{OptV} to deduce that there exists a measurable control $(\bar \w_0,\bar \w)(s)\in W$ for a.e. $s\in[0,S]$,  such that, setting $\varphi_0(s):= \int_0^s\bar \w_0(s')\,ds'$ for any $s\in[0,S]$ and considering the solution component $(\varphi,\tau,y)$ of our system associated with $(\bar \w_0,\bar \w)$,   up to a subsequence,  we have \footnote{We mean that all trajectories $(\tau_i,  y_i,\varphi_i)$ and $(\tau,y,\varphi)$ are extended to $[0,h+C]$ in a continuous and constant manner. In particular, $d\varphi/ds \equiv 0$ on $[S, h+C]$ and $d\varphi_i/ds \equiv 0$ on $[S_i, h+C]$.  Moreover, notice that $y_{-a_i}$ converges uniformly to $y_{-a}$ because $\tau_{1_i}$ converges uniformly to $\tau_1$ and $\xi_0$ is continuous on a compact set, hence uniformly continuous.}  
\bel{convergenze}
(\tau_i,  y_i,\varphi_i) \overset{L^\infty}{\to} (\tau,  y,\varphi), \  \ \left(\frac{d\tau_i}{ds}, \frac{d y_i}{ds}, \frac{d\varphi_i}{ds}\right) \overset{\text{weakly in $L^1$}
}{\to} \left(\frac{d\tau}{ds}, \frac{d y}{ds}, \frac{d\varphi}{ds}\right).
\eeq
In particular, $\varphi_0(S)=\tau_1(S)=\lim_i \tau_{1_i}(S_i)=h$ and $(S,\varphi_0,\varphi,\tau,y,\beta)$ is a (possibly, non canonical)  extended process,  when  $\beta$ is  the solution to
\[
\frac{d\beta}{ds}(s) = \left\| \frac{d\varphi}{ds}(s) \right\| \  \text{a.e. $s\in[0,S]$,} \ \ \beta_1(0)=0, \ \  \beta_l(0)=\beta_{l-1}(S) \ \text{for $l=2,\dots,N$.}
\]
In general, $\beta$  does not coincide with the uniform limit of the $\beta_i$'s. However, 
 \eqref{convergenze} implies that $\|d\varphi_l^j/ds \|_{L^1} \leq \underset{i}{{\rm lim\,inf}} \| d\varphi^j_{l_i}/ds \|_{L^1}$ for all $l=1,\dots,N$,  $j=1,\dots,m$, 
so that  
$$
\begin{array}{l}
\beta_N(S)= \sum_{l=1}^N[\beta_l(S)-\beta_l(0)]  
\leq \underset{i}{{\rm lim\,inf}} \sum_{l=1}^N\left(\sum_{j=1}^m \left\| \frac{d\varphi^j_{l_i}}{ds}\right\|_{L^1} \right) \\
\qquad\quad  = \underset{i}{{\rm lim\,inf}} \,\beta_{N_i}(S_i) = \underset{i}{{\rm lim\,inf}} \, v_i(T) \leq C.
  \end{array}
$$
Let $(x,v)$ be the g.c. solution associated with $(\varphi_{0},\varphi)$.  By  Thm. \ref{equivalence_th}, for each  $i$,    $(x_i,v_i)$ is the  g.c. solution corresponding to $(\varphi_{0_i},\varphi_i)$. Let $(\tilde\varphi ,\tilde\tau,\tilde y,\tilde\beta)$ be as in \eqref{exsol}. Let $\tilde\sigma$, $\tilde\sigma_i$ be the right inverses   of $\tilde\tau$, $\tilde\tau_i$, respectively. Define the countable set 
$
E:=\{t\in[0,T]: \ t \ \text{discontinuity point of $\tilde\sigma$ or  $\tilde\sigma_i$ for some $i\in\N$}\}.
$
We claim that
\bel{conv_sigma} 
\lim_{i\to+\infty}\tilde\sigma_i(t)=\tilde\sigma(t) \qquad \text{for any $t\in[0,T]\setminus E$.}
\eeq
Indeed, take $t\in[0,T]\setminus E$, set $s:=\tilde\sigma(t)$ and $s_i:=\tilde\sigma_i(t)$ and assume in contradiction that 
there is a subsequence, still denoted by $(s_i)$, such that $s_i\to s'\neq s$. Accordingly, we get $t=\tilde\tau_i(s_i)$ for any $i$, so that $t=\tilde\tau(s)=\lim_i \tilde\tau_i(s_i)=\tilde\tau(s')$, which implies that $\tilde\sigma$ is discontinuous at $t$, in contradiction with   the hypothesis  $t\notin E$.  Together with the uniform convergence of $\tilde y_i$ to $\tilde y$, this implies that 
\bel{convergenze2}
 \ds\lim_{i\to+\infty} x_i(t)=\lim_{i\to+\infty}\tilde y_i(\tilde\sigma_i(t))=\tilde y(\tilde\sigma(t))=  \tilde x(t)  \qquad \text{for all $t\in[0,T]\setminus E$.}
\eeq
Actually, this relation is always satisfied at $t\in\{0,T\}$, as $\tilde\sigma_i(0)=\tilde\sigma(0)=0$ for each $i$, while $\tilde\sigma_i(T)=NS_i\to NS=\tilde\sigma(T)$. Note that by Prop. \ref{equivalent_graph},(i), $(x,v)$ coincides with the g.c. solution associated with the canonical parameterization  $(S^c,\varphi_0^c,\varphi^c,\tau^c,y^c,\beta^c)$ of $(S,\varphi_0,\varphi,\tau,y,\beta)$, obtained by means of the increasing map $\delta_c$ defined as in Def. \ref{def_canonical}. At this point, Thm. \ref{equivalence_th},(ii)  guarantees that, corresponding to $(S^c,\varphi_0^c,\varphi^c,\tau^c,y^c,\beta^c)$, there exists an impulse control  $(\mu, \nu, \{\w^r\}_{r \in [0,h]})$  of which $(x,v)$ is the associated impulse solution.   

 It remains only to prove that $\mu_i\weak \mu$. Let $\sigma^c$ be the right inverse of $\tau^c$, so that by Def. \ref{def_canonical} one immediately has $\sigma^c =\delta_c\circ\sigma$. As a consequence, $\tilde\varphi^c\circ\tilde\sigma^c =\tilde\varphi\circ\tilde\sigma$, being $\tilde\varphi^c$ and $\tilde\sigma^c$ as in \eqref{exsol}. But by Thm. \ref{equivalence_th} one has that $\mu_i([0,t])=\tilde\varphi_i(\tilde\sigma_i(t))$ and $\mu([0,t])=\tilde\varphi^c(\tilde\sigma^c(t))$. The conclusion follows once again by \cite[Prop. 5.2]{VP} together with \eqref{convergenze} and \eqref{conv_sigma}.

  Assume now   $\{(w,\|w\|): w\in\K\}$  is convex. In this case,  the sets of velocities of the control system is convex, including the $\beta$ component. Hence, the above arguments apply to the whole sequence of solutions $(\varphi_i,\tau_i,y_i,\beta_i)$, allowing us to prove  also that $v_i$ converges to $v$ on $E$. \qed
\end{proof}

\section{Optimal Control Problems}\label{S_optimal}
Consider the   {\em optimal impulse control problem}
\bel{impulse_optimal}\tag{P$_{\rm imp}$}
\text{minimize } \ \ \Phi(x(T)),  
\ \ \text{ over $(\mu, \nu, \{\w^r\}_{ r \in [0,h]},x,v) \in \Gamma_{\rm imp}$},
\eeq
where $\Phi:{\mathbb R}^{n}\to{\mathbb R}$  is a lower semicontinuous, in short, l.s.c.,  function  and $\Gamma_{\rm imp}$ is the set of {\em feasible impulse processes}, defined as the  impulse processes for \eqref{DIS} satisfying the endpoint constraints
\bel{endpoint}
x(T)\in\T, \qquad
v(T) \leq C,
 \eeq
in which $C$ is a positive constant  and $\T\subset{\mathbb R}^{n}$ is a nonempty, closed set. 

\begin{proposition}\label{P_existence}
Assume $\K$ convex. If the optimization problem \eqref{impulse_optimal} has at least one feasible process, then it admits a minimizer. 
\end{proposition}
\begin{proof}
 Take a minimizing sequence $(\mu_i, \nu_i, \{\w_i^r\}_{ r \in [0,h]}, x_i, v_i)$ of feasible impulse processes for problem  \eqref{impulse_optimal}. From Prop. \ref{robust_th} it follows that, up to a subsequence, there exists an impulse process $(\mu, \nu, \{\w^r\}_{ r \in [0,h]},x,v)$ for \eqref{DIS}  such that $x(T)=\underset{i\to+\infty}\lim x_i(T)\in\T$ and $v(T)\le C$. Hence, $(\mu, \nu, \{\w^r\}_{ r \in [0,h]},x,v)$ is feasible and actually a minimizer, as the function $\Phi$ is l.s.c..\qed
%
 \end{proof}
\begin{remark} In view of  Prop. \ref{robust_th},   this existence result holds even if the final constraint $x(T)\in\T$  is replaced by  $(x,v)(T)\in\tilde\T$,  for any   $\tilde\T\subset{\mathbb R}^n\times[0,+\infty)$ closed, and   $\Phi(x)$ is replaced by a l.s.c. function $\tilde\Phi(x,v)$, increasing w.r.t.  the $v$-variable. Actually, again by Prop. \ref{robust_th},  we can also disregard the latter monotonicity condition when  the set  $\{(w,\|w\|): \  w\in \K\}$  is convex.
\end{remark}

In order to derive a Maximum Principle,  we associate with  problem \eqref{impulse_optimal}, the following, non impulsive,  {\em extended optimal control problem} 
\bel{extended_optimal}\tag{P$_{\rm ext}$}
\text{minimize } \ \ \Phi(y_N(S)), 
\ \  \text{over $(S, \varphi_0,\varphi,\tau, y, \beta)\in\Gamma_{\rm ext}$,}
\eeq
where $\Gamma_{\rm ext}$ is  the set of {\em feasible extended processes}, defined as  the extended processes $(S, \varphi_0,\varphi,\tau, y, \beta)$ such that
\bel{endpoint_ext}
y_N(S) \in \T,  \qquad \beta_N(S)\leq C.
\eeq
 Thm. \ref{equivalence_th} directly implies that \eqref{impulse_optimal} has a feasible process if and only if the same holds for \eqref{extended_optimal} and the infima of \eqref{impulse_optimal} and \eqref{extended_optimal} coincide.
 To highlight the crucial role of  Thm. \ref{equivalence_th} in determining necessary conditions of optimality  for  \eqref{impulse_optimal}, we first state a Maximum Principle for the extended optimal control problem.  For this purpose,  we introduce
the {\em (unmaximized) Hamiltonian} 
\[
\begin{split}
&H(\{t_k\},\{x_k\}, \{q_k^0\},\{q_k\}, d,w_0,w):= w_0  \sum_{l=1}^N q_l^0\\
& + \sum_{l=1}^N q_l\cdot \Big[F_l(\{t_k\},\{x_{k}\}) w_0 + \sum_{j=1}^m G_{j_l}(\{t_k\},\{x_{k}\}) w_l^j \Big] + d \sum_{l=1}^N \|w_l\|,
\end{split}
\]
and define the {\em drift Hamiltonian} $H^{\rm (dr)}$ and {\em impulse Hamiltonian} $H^{\rm (imp)}$, as 
\[
\begin{split}
&H^{\rm (dr)}(\{t_k\},\{x_k\}, \{q_k^0\},\{q_k\}):=  \sum_{l=1}^N q_l^0 +  \sum_{l=1}^N q_l\cdot F_l(\{t_k\},\{x_{k}\}),   \\
&H^{\rm (imp)}(\{t_k\},\{x_k\}, \{q_k^0\},\{q_k\}, d):= \max_{\underset{\|w\|=1 \ }{w\in\K^N,}} \sum_{l=1}^N q_l\cdot  \sum_{j=1}^m G_{j_l}(\{t_k\},\{x_{k}\}) w_l^j   +d,
\end{split}
\]
where   $F_l(\{t_k\},\{x_k\})=F_l(\{t_k\}_{k=l}^{M},\{x_k\}_{k=1}^l):=f(t_l,x_l,x_{l-1},\dots, x_1, \xi_0(t_M-Mh), \xi_0(t_{M-1}-Mh),\dots, \xi_0(t_l-Mh))$ for  $l=1,\dots, M$, and $F_l(\{t_k\},\{x_k\})=F_l(t_l,\{x_k\}_{k=l-M}^l):=f(t_l,x_l,x_{l-1},\dots, x_{l-M})$  for $l=M+1,\dots,N$. The maps $G_{j_l}$ are defined analogously, with $g_j$ in place of $f$.\footnote{Note that in system \eqref{aux_system}, $y_0(s)=\xi_0(\tau_1(s)-h)=\xi_0(\tau_M(s)- Mh)$, $y_{-1}(s)=\xi_0(\tau_1(s)-2h)=\xi_0(\tau_{M-1}(s)- Mh)$, \dots,  $y_{-(M-1)}(s)=\xi_0(\tau_1(s)-Mh)$, as $\tau_l(s)=\tau_1(s)+(l-1)h$.}
Finally, set
$$
\W:=\Big\{(w_0,w)\in[0,+\infty[\times\K^N \text{ : } w_0+\sum_{l=1}^N \|w_l\|=1\Big\}.
$$
In the following, given a reference extended process $(\bar S, \bar\varphi_0,\bar\varphi,\bar\tau,\bar y, \bar\beta)$ and a function $z$ defined on  ${\mathbb R}\times{\mathbb R}^{(M+1)n}$,   for any $s\in[0,\bar S]$, $i=1,\dots,N$,  we set $\hat z^i(s):=z(\bar\tau_i(s),\{\bar y_{i-k}(s)\}_{k=0}^M)$, and write $\frac{\partial\hat z^i}{\partial t}(s)$, and $\nabla_{x_k}\hat z^i(s)$  for  the partial derivative of $z^i$ w.r.t. $t$, and the Jacobian of $z^i$ w.r.t. the $k$-th delay state variable ($k=1,\dots,M+1$)  evaluated at  $(\bar\tau_i(s),\{\bar y_{i-k}(s)\}_{k=0}^M)$, respectively.


\begin{theorem} \label{pmp_extended}
Assume that $\xi_0$ and $\Phi$ are $C^1$ functions and let $(\bar S, \bar\varphi_0,\bar\varphi,\bar\tau,\bar y, \bar\beta)$ be a {\em minimizer} for \eqref{extended_optimal}. Then there exist numbers $\lambda\geq0$ and $d\leq 0$ ($d=0$ if $\bar\beta_M(\bar S)<C$), a subset $\NN\subset[0,\bar S]$ of null Lebesgue measure, and absolutely continuous functions $(q^0_l,q_l):[0, \bar S]\to{\mathbb R}^{1+n}$ for $l=1,\dots,N$ satisfying the following relations:
\vsm
\noindent {\rm (A): }\ $\|(q_1,\dots,q_N)\|_{L^\infty}+\lambda\neq 0$;
\vsm
\noindent
{\rm (B): } for l=1,\dots,N and a.e. $s\in[0,\bar S]$,  for $\mathbb{I}_{l}^M:=1$ if $l\le M$ and $0$ otherwise,    
{\small \[
\begin{array}{l}
\ds-\frac{d q_l^0}{ds}(s) = q_l(s)\cdot \Big[\frac{\partial \hat f^l}{\partial t}(s) \frac{d\bar\varphi_0}{ds}(s)  + \sum_{j=1}^m  \frac{\partial \hat  g^l_{j}}{\partial t}(s)  \frac{d\bar\varphi_l^j}{ds}(s)  \Big]+ \mathbb{I}_{l}^M\sum_{i=1}^l q_{l-i+1}(s)\\
 \ds  \cdot\Big[ \nabla_{x_{M+2-i}} \hat f^{l-i+1}(s)\frac{d\bar\varphi_0}{ds}(s) + \sum_{j=1}^m \nabla_{x_{M+2-i}}\hat g^{l-i+1}_{j} (s) \frac{d\bar\varphi_{l-i+1}^j}{ds}(s)\Big] \frac{d\xi_0}{dt} (\bar\tau_l(s)-Mh), \\
\ds-\frac{d q_l}{ds}(s) = \sum_{i=l}^{(l+M)\land N} q_{i}(s) \cdot\Big[\nabla_{x_{i-l+1}} \hat f^i(s) \frac{d\bar\varphi_0}{ds}(s) + \sum_{j=1}^m \nabla_{x_{i-l+1}}\hat g^l_{j}(s) \frac{d\bar\varphi_i^j}{ds}(s) \Big].
\end{array}
\]}
 {\rm (C): }\ \,  $-q_N(\bar S)\in\lambda\nabla\Phi(\bar y_N(\bar S)) + N_{\T}(\bar y_N(\bar S))$ and, for any $l=2,\dots,N$, 
 $(q_l^0(0),q_l(0))=(q_{l-1}^0 (\bar S), q_{l-1}(\bar S))$.
\bel{zero1}
\begin{split} 
{\rm (D):} \ 0 &=\max_{(w_0,w)\in \W} H(\{\bar\tau_l(s)\},\{\bar y_l(s)\}, \{q_l^0(s)\},\{q_l(s)\}, d,w_0,w)  \qquad\qquad \ \ \ \ \\
&=\max\Big\{ H^{\rm (dr)}(\{\bar\tau_l(s)\},\{\bar y_l(s)\}, \{q_l^0(s)\},\{q_l(s)\}), \\
& \qquad H^{\rm (imp)}(\{\bar\tau_l(s)\},\{\bar y_l(s)\}, \{q_l^0(s)\},\{q_l(s)\}, d) \Big\}, \ s\in[0,\bar S].
\end{split}
\eeq
In particular, for all $l=1\dots,N$ and any $s\in[0,\bar S]$ we have
\bel{zero2}
\max_{w\in\K, \, \|w\|=1} q_l(s)\cdot \sum_{j=1}^m \hat g^l_j(s) w^j+d \leq 0.
\eeq
%
%
\bel{equal0} 
\begin{array}{l}
{\rm (E):} \quad H\Big(\{\bar\tau_l(s)\},\{\bar y_l(s)\}, \{q_l^0(s)\},\{q_l(s)\}, d,\frac{d\bar\varphi_0}{ds}(s) ,\frac{d\bar\varphi}{ds}(s)\Big)  \\
\ds \quad= \max_{(w_0,w)\in \W} H(\{\bar\tau_l(s)\},\{\bar y_l(s)\}, \{q_l^0(s)\},\{q_l(s)\}, d,w_0,w), \ \ s\in[0,\bar S]\setminus\NN.
\end{array}
\eeq
 In particular, for all $s\in[0,\bar S]\setminus\NN$ we have
\bel{zero3}
\Big\|\frac{d\bar\varphi_l}{ds}(s) \Big\|>0 \implies \max_{w\in\K, \, \|w\|=1} q_l(s)\cdot \sum_{j=1}^m \hat g^l_j(s) w^j+d = 0.
\eeq
%
%
%
\bel{hamdrift}
{\rm (F):} \quad\frac{d\bar\varphi_0}{ds}(s) \Big[\sum_{l=1}^N q_l^0(s) +  \sum_{l=1}^N q_l(s) \cdot \hat f^l(s) \Big] =0, \quad s\in[0,\bar S]\setminus\NN,\qquad
\eeq
\bel{hamimp}
\sum_{l=1}^N q_l(s) \cdot \Big( \sum_{j=1}^m \hat g^l_j(s) \frac{d\bar\varphi_l^j}{ds}(s) \Big) + d \sum_{l=1}^N \Big\|\frac{d\bar\varphi_l}{ds}(s) \Big\|=0,  \quad s\in[0,\bar S]\setminus\NN.
\eeq
In particular, for all $s\in[0,\bar S]\setminus\NN$ we have
\bel{implica}
H^{\rm (dr)}(\{\bar\tau_l(s)\},\{\bar y_l(s)\}, \{q_l^0(s)\},\{q_l(s)\})<0 \implies \frac{d\bar\varphi_0}{ds}(s)=0.
\eeq
Moreover, for all $s\in[0,\bar S]\setminus\NN$ s.t. $H^{\rm (dr)}(\{\bar\tau_l(s)\},\{\bar y_l(s)\}, \{q_l^0(s)\},\{q_l(s)\})<0$,  \bel{hamdrift2}
\begin{split}
&\sum_{l=1}^N q_l(s) \cdot \Big( \sum_{j=1}^m \hat g^l_j(s) \frac{d\bar\varphi_l^j}{ds}(s) \Big) + d \\
&\qquad\qquad\qquad\qquad =  H^{\rm (imp)} (\{\bar\tau_l(s)\},\{\bar y_l(s)\}, \{q_l^0(s)\},\{q_l(s)\}, d) =0.
\end{split}
\eeq
Similarly, for all $s\in[0,\bar S]\setminus\NN$, $H^{\rm (imp)} (\{\bar\tau_l(s)\},\{\bar y_l(s)\}, \{q_l^0(s)\},\{q_l(s)\}, d)<0$ implies $\frac{d\bar\varphi}{ds}(s)=0$ and 
\bel{hamimp2}
\begin{split}
H^{\rm (dr)} (\{\bar\tau_l(s)\},\{\bar y_l(s)\}, \{q_l^0(s)\},\{q_l(s)\}) =0.
\end{split}
\eeq
\end{theorem}
The proof of this theorem and the next are given in Appendix \ref{proof_pmpi}.

In the following, for a fixed impulse process $(\bar \mu,  \bar\nu,\{\bar \w^{r}\}_{ r \in [0,h]},\bar x,\bar v)$ and for any function $h\in C^1({\mathbb R}\times{\mathbb R}^{(M+1)n},\R^n)$,  we set  $ \bar h (t):= h(t,\{ \bar x (t-kh)\})$ and write $\frac{\partial \bar h}{\partial t}(t)$, $\nabla_{x_k}\bar h(t)$  for the partial derivative  w.r.t. $t$ and the Jacobian   w.r.t. the $k$-th delay state variable ($k=1,\dots,M+1$) of $h$ at $(t,\{ \bar x (t-kh)\})$, respectively.    
\begin{theorem} \label{main:th}
Assume  $\xi_0$,  $\Phi\in C^1$ and let $(\bar \mu,  \bar\nu,\{\bar \w^{r}\}_{ r \in [0,h]},\bar x,\bar v)$ be a minimizer for \eqref{impulse_optimal}. Then, there exist constants $\lambda\geq0$ and $d\leq0$ ($d=0$ if $\bar v(T)<C)$ and paths $(p^0,p)\in BV([0,T], {\mathbb R}\times{\mathbb R}^n)$, right continuous on $]0,T]$, with the following properties:

\noindent {\rm (i)} $\lambda + \|p\|_{L^{\infty}(0,T)}\neq 0$;
\vsm
\noindent {\rm (ii)} $(p^0,p)(Nh)=(\alpha_N^h,\eta_N^h)(1)$,  $(p^0,p)(lh)=(\alpha_{l+1}^0,\eta_{l+1}^0)(1)$ if $l=1,\dots,N-1$,
 {\small\bel{eq_p}
\begin{split}
p(t)&=\eta_l^0(1)-\sum_{i=0}^M \int_{(l-1)h}^t  p(t'+ih)\cdot \nabla_{x_{i+1}} \bar f(t'+ih) dt' \\
&-\sum_{i=0}^M\sum_{j=1}^m \int_{[(l-1)h,t ]} p(t'+ih) \cdot   \nabla_{x_{i+1}} \bar g_j(t'+ih)  d(\bar\mu^c)^j(t'+ih) \\
&- \sum_{r\in]0,t-(l-1)h]} [\eta_l^r(0)-\eta_l^r(1)], \qquad t\in](l-1)h,lh[, \ \ l=1,\dots,N, \qquad\qquad
\end{split}
\eeq
 \bel{eq_p01}
\begin{split}
&p^0(t)= \alpha_l^0(1) - \int_{(l-1)h}^t  p(t') \cdot \frac{\partial \bar f}{\partial t}(t') dt'- \sum_{j=1}^m \int_{[(l-1)h,t ]}^t p(t') \cdot   \frac{\partial\bar g_j}{\partial t}(t') d(\bar\mu^c)^j(t')   \\
&\   - \chi_{[0,Mh]}(t)\sum_{i=0}^{l-1}\left[ \int_{(l-1)h}^t p(t'-ih)\cdot \nabla_{x_{M+1-i}} \bar f(t'-ih)\cdot     \frac{d\xi_{0}}{dt}(t'-Mh)  dt' \right. \\
&\left.\ + \sum_{j=1}^m \int_{[(l-1)h,t ]} p(t'-ih)\cdot \nabla_{x_{M+1-i}} \bar g_j(t'-ih) \cdot \frac{d\xi_{0}}{dt}(t'-Mh) d(\bar\mu^c)^j(t'-ih)\right]\\
&\  -\sum_{r\in]0,t-(l-1)h]} [\alpha_l^r(0)-\alpha_l^r(1)],  \quad t\in](l-1)h,lh[, \ \ l=1,\dots,N,
\end{split}
\eeq}
where, for any interval $I$ the map $\chi_I(\cdot)$ denotes the characteristic function of $I$  and
 $(p^0(t),p(t))=0$ for any $t>T$. For any $l=1,\dots,N$,  $r\in[0,h]$,  the  functions $\eta_l^r:[0,1]\to {\mathbb R}^n$ and $\alpha_1^r:[0,1]\to{\mathbb R}$ above  satisfy, respectively, 
{\small \bel{eq_eta}
\left\{\begin{array}{l}
\ds-\frac{d\eta_l^r}{ds}(s) = \sum_{i=l}^{(l+M)\wedge N} \eta_i^r(s) \cdot \Big( \sum_{j=1}^m \nabla_{x_{i-l+1}} g_j(r+(i-1)h, \{\bar\zeta^r_{i-k}(s)\})(\bar \w_i^r)^j(s) \Big) , \\ 
\eta_l^0(0)=\eta_{l-1}^h(1), \quad \eta_l^r(0)=
p^-(r+(l-1)h) \ \ \text{if $r\in]0,h]$},  
\end{array}\right.
\eeq}
{\small \bel{eq_alpha1}
\left\{\begin{array}{l}
\ds-\frac{d\alpha_l^r}{ds}(s) =  \eta_l^r(s) \cdot \Big( \sum_{j=1}^m \frac{\partial g_j}{\partial t}(r+(l-1)h, \{\bar\zeta^r_{l-k}(s)\}) (\bar \w_l^r)^j(s)\Big)  +\mathbb{I}_l^M \sum_{i=1}^l \eta_{l-i+1}^r(s)\\
 \ds  \cdot  \Big( \sum_{j=1}^m \nabla_{x_{M+2-i}}g_j(r+(l-i)h, \{\bar\zeta_{l-i+1-k}^r(s)\}) (\bar \w_{l-i+1}^r)^j (s)\Big)  \cdot 
\frac{d\xi_{0}}{dt}(r-(M+1-l)h), \\
\alpha_l^0(0)=\alpha_{l-1}^h(1), \quad  \alpha_l^r(0)=
(p^0)^-(r+(l-1)h) \ \  \text{if $r\in]0,h]$},
\end{array}\right.
\eeq}
Here, $\bar\zeta_l^r$, $\bar\theta_l^r$ are the functions associated with  $(\bar \mu,  \bar\nu,\{\bar \w^{r}\}_{ r \in [0,h]},\bar x,\bar v)$  defined as in \eqref{jump_x1},   \eqref{jump_v1}, respectively, and
$(p^0(0),p(0)):=(\alpha_0^h(1),\eta_0^h(1))$.
\vsm
 \noindent {\rm (iii)} $-p(T) \in \lambda \nabla\Phi(\bar x(T)) + N_\T(\bar x(T))$;
\vsm
 \noindent {\rm (iv)} $\ds \max_{w\in\K, \, \|w\|=1} p(t)\cdot \sum_{j=1}^N \bar g_j(t) w^j +d \leq 0$ for all $t\in[0,T]$;
\vsm
\noindent {\rm (v)} $\ds {\rm supp}(\bar\nu) \subset \Big\{t\in[0,T] \text{ : } \max_{w\in\K, \, \|w\|=1} p(t)\cdot \sum_{j=1}^N \bar g_j(t) w^j +d = 0 \Big\}$;
\vsm
\noindent {\rm (vi)}  $\ds \sum_{l=1}^N [p^0(r+(l-1)h)+p(r+(l-1)h)\cdot \bar f(r+(l-1)h)]=0$ for any $r\in]0,h[$. Moreover, this equality holds at $r=0$ if $\sum_{l=1}^N\int_0^1 \| \bar \w_l^0(s)\| ds=0$ and at $r=h$ if $\sum_{l=1}^N\int_0^1 \big(\| \bar \w_l^0(s)\| + \| \bar \w_l^h(s)\| \big)ds=0$.
\vsm
\noindent {\rm(vii)} For any $r\in[0,h]$ such that $ \sum_{l=1}^N\int_0^1 \|\bar \w_l^r(s)\| ds >0$,  for a.e. $s\in[0,1]$,
\[
\begin{split}
&\sum_{l=1}^N \alpha_l^r(s) + \sum_{l=1}^N\eta_l^r(s) \cdot f(r+(l-1)h, \{\bar\zeta^r_{l-k}(s)\}) \\
&\ \  \leq \max_{l=1,\dots,N} \Big[ \max_{w\in\K, \, \|w\|=1} \eta_l^r(s) \cdot \sum_{j=1}^m g_j(r+(l-1)h, \{ \bar\zeta_{l-k}^r(s)\}) w^j +d\Big] =0.
\end{split}
\]
\noindent{\rm (viii)} For a.e. $s\in[0,1]$ such that $\|\bar \w_i^r(s)\|>0$ for some $r\in[0,h]$ and $i$,  
\[
\begin{split}
&\max_{l=1,\dots,N} \Big[ \max_{w\in\K, \, \|w\|=1} \eta_l^r(s) \cdot \sum_{j=1}^m g_j(r+(l-1)h, \{ \bar\zeta_{l-k}^r(s)\}) w^j +d\Big]\\
&\qquad \qquad =\max_{w\in\K, \, \|w\|=1} \eta_i^r(s) \cdot \sum_{j=1}^m g_j(r+(i-1)h, \{ \bar\zeta_{i-k}^r(s)\}) w^j +d.
\end{split}
\]
\end{theorem}
Some comments are appropriate. 
Condition (ii) is an adjoint equation for $(p_0,p)$ in integral form, having the character of an `advance' ODE, consistent with the standard adjoint equation in optimality conditions for non impulse,  delay optimization problems. Note that the `attached'  controls $\bar \w^r_i$ influence the jumps in $(p_0,p)$, which may occur at times when $\bar\mu$ (and $\bar x$) are continuous. Indeed, both jump equations \eqref{eq_eta} and \eqref{eq_alpha1}  involve  attached controls of several times. Note also that  $p_0$ equation depends on the derivative of the state history function  $\xi_0$.
Properties (v), (vi) correspond to the `constancy of the Hamiltonian condition'. Specifically, the equality in condition (vi) holds at $r=0$ if $t=0$  is not an atom for $\bar\nu$, and at $r=h$ if none of the points $t=ih$, $i=1, \dots, N$, are atoms for $ \bar\nu$.
Condition (v), which pertains to the maximality condition, combines   with  (viii) to provide insights into the support location of the measure $\bar\nu$. It specifically outlines the characteristics of the attached controls that govern the instantaneous state evolution at the atoms of $\bar\nu$.

\section{Conclusions}\label{S_concl}
We have introduced a notion of an impulse process for an impulse control system with time delays, generalizing the method of graph completions originally developed for impulse systems without delays. In particular, the main result of the paper, established in Theorem \ref{equivalence_th},  is an equivalent formulation of  impulse controls and solutions as suitably defined graph completions and graph completion solutions, respectively.  Using these concepts, we have shown existence and uniqueness of the impulse solution associated with an impulse control, the density of the set of trajectories associated with absolutely continuous controls in the set of impulse trajectories, and a compactness result of any subset of impulse solutions associated with BV  controls of equibounded   total variation (when $\K$ is convex).  Finally,  we have obtained a maximum  principle for an associated Mayer problem with endpoint constraints. 
 
The generalization of the  graph completions  technique  to  impulse control problems with time delays  paves the way  for further research in several directions,  such as considering  more general delayed impulse optimization problems with state constraints (to extend, for instance, the results in  \cite{AKP12,AKP19,MS20} to the case with delays);   determining sufficient conditions for the non-occurence of a gap between  the minimum of an impulse problem with time delays and the infimum of the same payoff  over absolutely continuous controls, as done e.g.  in  \cite{MRV,FM1,FM2,FM3,MPR} for the delay-free  case; determining stabilizability criteria in the   fashion of  \cite{LM19,LM21,LM22} when time delays are present. Another  promising research direction would be the development  of  sufficient optimality conditions  expressed in terms of solutions to the highly non-standard  Hamilton-Jacobi-Bellman  equations and boundary conditions  that arise in optimal delay impulse control. There is also potential for developing new nonlinear models of fed-batch fermentation or of delayed neural networks with impulse controls, just to give some examples. Indeed,  there is a wide literature on delayed impulse control systems for such  applications (see e.g. \cite{XSSZ02,GLFX06,LCH20} and references therein), relying, however, on a different model from ours,  in which  impulse controls are given by a finite or countable number of jump instants with a preassigned jump-function.


\appendix  
\section{Proof of Theorem \ref{equivalence_th}} \label{proof_mainth} Initially we recall if $A:[T_1,T_2]\to[S_1,S_2]$ is a (possibly not strictly) increasing function such that $A(T_1)=S_1$, $A(T_2)=S_2$, and $A$ is right continuous on $]T_1,T_2[$, $B$ is its right inverse,   and $F:[T_1,T_2]\to{\mathbb R}^n$ is a Borel function, then   (see \cite[Lemmas 2.4, 2.6]{MiRu})
\bel{change_var}
\int_{[B(s_1),B(s_2)]}F(t) dA(t)=\int_{[s_1,s_2]}F(B(s)) ds   \quad\text{for every $S_1\le s_1<s_2\le S_2$.}
\eeq
(i) Let  $(\mu,\nu,\{\w^r\}_{r\in[0,h]},x,v)$ be an impulse process and define $(S, \varphi_0,\varphi,\tau,y,\beta)$  as in Thm. \ref{equivalence_th},(i). By  Def. \ref{imp_control},(i) ,  \eqref{mu_i_rel}, and  \eqref{normalize}  we deduce that   $(S,  \varphi_0, \varphi)$ is a canonical extended control such that $d\varphi_0(s)/ds=0$ a.e.  $s\in\Sigma^r$ (notice in particular that $d\varphi_0(s)/ds=m_0(\varphi_0(s))$ a.e. by \eqref{change_var} applied to $A=\phi$, $B=\varphi_0$ and $F=m_0$).  By \eqref{jump_x1}  and \eqref{jump_v1}  we have  that, for any  $r\in\J$ and $l =1,\ldots, N$,   $(y_l,\beta_l)$  satisfies 
\bel{jump_ode3}
\begin{cases}
\ds\frac{dy_l}{ds}(s)= \sum_{j=1}^m g_j\left((l-1)h+r, \{y_{l-k}(s)\}_{k}\right) \frac{d\varphi_l^j}{ds}(s)  \quad   \mbox{ a.e. \  } s \in \Sigma^r, \\
\ds\frac{d\beta_l}{ds}(s)= \left\|\frac{d\varphi_l}{ds}(s)\right\| \quad  \mbox{ a.e. \  } s \in \Sigma^r, \\
(y_l,\beta_l)  (\phi^- (r)) = 
(\zeta^r_l, \theta^r_l)(0).
\end{cases}
\eeq
Take  $l=1,\ldots,N$, 
$s \in ]0,S[$ and 
let $t =\tau_l(s)$. 
The  properties of $(y_l,\beta_l)$ on the $\Sigma^r$'s yield
\[
\begin{split}
&x(t)=   \zeta^0_l (1) +\int_{(l-1)h}^t  f(t', \{x(t'-kh\})dt' 
+ \int_{[(l-1)h,t]} \sum_{j=1}^mg_j(t', \{x(t'-kh\})(\mu^{c})^j(dt') \\
&\ \quad\qquad\qquad+  \underset{ \{r\in\J\,:\, r\in]0,t-(l-1)h] \}}{\sum}  \left[y_l (\phi^+(r) )- y_l (\phi^- (r) )\right], \\
&v(t)=   \theta^0_l (1)  + \int_{[(l-1)h,t]} |\mu^{c}|(dt')
+ \underset{ \{r\in\J\,:\, r\in]0,t-(l-1)h] \}}{\sum}  \left[\beta_l (\phi^+(r) )- \beta_l (\phi^- (r) )\right].
\end{split}
\]
 But, in consequence of  \eqref{change_var} and arguing as in \cite{FMV1} one has
\[
\begin{split}
&\int_{(l-1)h}^t  f(t', \{x(t'-kh\})dt' = \int_{[0, s]}  f(\tau_l(s'),\{y_{l-k}(s')\})\frac{d\varphi_0}{ds}(s')\,ds', \\
& \int_{[(l-1)h,t]} \sum_{j=1}^mg_j(t', \{x(t'-kh\})(\mu^{c})^j(dt')=  \\
&\qquad\qquad\qquad\qquad=\int_{[0, s]}  \sum_{j=1}^mg_j(\tau_l(s'), \{y_{l-k}(s')\})\frac{d\varphi^j_l}{ds}(s')  \chi_{[0,S]\backslash \cup_{r\in\J}\Sigma^r}(s')ds'\,, \\
& \int_{[(l-1)h,t]}  |\mu^{c}|(dt')=\int_{[0, s]} \left\|\frac{d\varphi_l}{ds}(s')\right\|  \chi_{[0,S]\backslash \cup_{r\in\J}\Sigma^r}(s')ds'\,.
\end{split}
\]
%
If $s\in]0,S[\setminus  \cup_{r\in\J}\Sigma^r$,  observing that  $\zeta^0_l (1)=y_l(0)+ (y_l (\phi^+(0) )- y_l(\phi^- (0) ))$  and $\theta^0_l (1)=\beta_l(0)+ (\beta_l (\phi^+(0) )- \beta_l(\phi^- (0) ))$,   from the above relations it follows, for each $l=1,\dots,N$,
\begin{equation}
\label{y_eqn}
\begin{array}{l}
\ds y_l(s) =y_l (0)  + \int_{[0, s]}  f(\tau_l(s'), \{y_{l-k}(s')\})\frac{d\varphi_0}{ds}(s') ds'  \\
 \ds \  +\int_{[0, s]}  \sum_{j=1}^mg_j(\tau_l(s'), \{y_{l-k}(s')\})\frac{d\varphi^j_l}{ds}(s') ds' , \quad
  \ds \beta_l(s) =\beta_l (0)   +
  \int_{[0, s]}  \left\|\frac{d\varphi_l}{ds}(s') \right\|ds'.
  \end{array} 
 \end{equation}
If instead  $s \in ]0,S[\cap \Sigma^r$ for some  $r\in\J$,  since $\phi^- (r)$ is a density point of $[0,S]\backslash \cup_{r\in\J}\Sigma^r$ and the $(y_l,\beta_l)$'s are continuous,  the relations  in   (\ref{y_eqn}) are valid at $s^-:=\phi^- (r)$ for  each $l$.
But then,  
in view of \eqref{jump_ode3} and since  $d\varphi_0/ds=0$  a.e. on $\Sigma^r$,
we see that  \eqref{y_eqn} is satisfied for arbitrary $s \in ]0,S[$\,.  By definition, $(y_l,\beta_l)(S)= (\zeta_l^h,\theta_l^h)(1)$. Hence, if $S\notin \cup_{r\in\J}\Sigma^r$ or $l=N$,  then $(y_l,\beta_l) $ satisfies \eqref{y_eqn} on $]0,S]$ by continuity.  If, on the other hand,  $l\in\{1,\dots,N-1\}$ and $S \in \Sigma^r$  for some   $r\in\J$, then $r  =h$ and we derive that \eqref{y_eqn} still holds on  $]0,S]$ arguing as above, by \eqref{jump_ode3}.   
 Furthermore, note that the  $(y_l,\beta_l)$'s satisfy the boundary conditions in \eqref{aux_system} by \eqref{cond_finy}.
 Reviewing the preceding relations,    $(S,\varphi_0,\varphi, \tau, y, \beta)$ is a canonical extended process.  Moreover,  \eqref{equiv_tr2} implies that $(x,v)(t)=(\tilde y,\tilde\beta)(\tilde\sigma(t))$ for all $t\in](l-1)h,lh[$, $l=1,\dots,N$,  such that 
 $r:=t- (l-1)h\notin\J$. This is actually true for all $t\in[0,T]$, because of the right continuity of $(x,v)$ and $\tilde\sigma$. 
In conclusion, assertion (i)  has been confirmed.

 \vsm
  
(ii)   Fix a canonical extended process  $(S,  \varphi_0, \varphi, \tau, y, \beta)$  and let $(\mu,\nu,\{\w^r\}_{r\in[0,h]},x,v)$ be as in  Thm. \ref{equivalence_th},(ii).  To prove that $\nu\in\V_c(\mu)$, note that $\nu^c=|\mu^c|$   by \cite[Thm. 6.1]{FM224}. Furthermore, by \cite[Thm. 5.1]{AR15},   the right inverse $\sigma$ of $\varphi_0$ can be point-wisely approximated by  a sequence of   strictly increasing, absolutely continuous, onto  functions $\sigma_i:[0,h]\to[0,S]$  with  strictly increasing, 1-Lipschitz continuous inverse functions   $\varphi_{0_i}$  which converge uniformly to $\varphi_0$. 
For every $i\in\N$,  let  $(S,  \varphi_{0_i}, \varphi, \tau_i, y_i, \beta_i)$ be the  extended process associated with   $(S,  \varphi_{0_i}, \varphi)$ and define $(\tilde\varphi,\tilde\tau,\tilde y, \tilde \beta)$ and  $(\tilde\sigma,\tilde\sigma_i, \tilde \tau_i)$ associated to  $(\varphi, \tau, y, \beta)$ and  $(\tilde\sigma,\sigma_i, \tau_i)$, respectively,  according to \eqref{exsol}.
Note that  $\tilde\sigma_i(T)=NS$, $\tilde\tau_i(NS)=T$,    $\tilde\tau_i$ coincides with $\tilde\sigma_i^{-1}$  and   converges uniformly to $\tilde\tau$, while $\tilde\sigma_i(t)\to\tilde\sigma(t)$  for every $t\in]0,T]\setminus\{h,\dots,(N-1)h\}$. Consider the sequence of absolutely continuous measures $\mu_i$ given by 
$$
\mu_i([0,t]):= \int_0^{\tilde\sigma_i(t)}\frac{d\tilde\varphi}{ds}(s)\,ds=\tilde\varphi(\tilde\sigma_i(t)), \quad\text{for $t\in[0,T]$,}
$$
so that
$|\mu_i|([0,t])= \int_0^{\tilde\sigma_i(t)}\left\|\frac{d\tilde\varphi}{ds}(s)\right\|\,ds=\tilde\beta(\tilde\sigma_i(t))$
for $t\in[0,T]$. Clearly,  $\lim_i\mu_i([0,t])=\mu([0,t])$ and $\lim_i|\mu_i|([0,t])=\nu([0,t])$ for all $t\in]0,T]\setminus\{h,\dots,(N-1)h\}$. Hence, from  \cite[Prop. 5.2]{VP} it follows  that $(\mu_i,|\mu_i|) \weak(\mu,\nu)$. The proof that $\nu\in\V_c(\mu)$ is thus complete.  

To show that $\{\w^r_l\}_{r\in[0,h]}$ is attached to $(\mu,\nu)$, observe that  Def.  \ref{imp_control}, (i)  holds since the extended control $(S,\varphi_0,\varphi)$ is canonical, while  \ref{imp_control},(ii)-(iii) can be derived from \eqref{u_v}, since  
\begin{equation}\label{2_2_0tilde}
 \tilde\sigma^- (r+(l-1)h)= \sigma^- (r)+(l-1)S, \quad \tilde\sigma^+ (r+(l-1)h)= \sigma^+ (r)+(l-1)S,
  \end{equation}
for any $i=1,\dots,N$ and  $r\in]0,h[$, and 
  \begin{equation}\label{2_2_0ih}
  \begin{array}{l}
  \tilde\sigma^- (0)= \sigma^- (0)=0, \quad \tilde\sigma^+ (0)= \sigma^+ (0), \\[1.5ex]
   \tilde\sigma^- (lh)= \sigma^- (h)+(l-1)S,  \quad \tilde\sigma^+(lh) = \sigma^+ (0)+lS, \quad l=1,\dots,N-1,
    \\[1.5ex]
    \tilde\sigma^- (Nh)= \sigma^- (h)+(N-1)S, \quad \tilde\sigma^+ (Nh)= \sigma^+ (h)+(N-1)S=NS.
 \end{array}
  \end{equation}
It only remains to prove that the g.c. solution $(x,v)= (\tilde y,\tilde\beta)\circ\tilde\sigma$  is an impulse solution associated with the impulse control $(\mu,\nu,\{\w^r\}_{r\in[0,h]})$.  For this purpose, define the function  $\phi$ as in \eqref{phi}  and  let $\hat\varphi_0$ be its right inverse. Note that $\phi(h)=h+\sum_{l=1}^N\int_0^h\nu_l(dt)=h+\nu([0,T])=S$ by $\nu_l$'s definition, since  $(S,\varphi_0,\varphi)$ is canonical.
We claim that
$
\varphi_0\equiv\hat\varphi_0.
$
First, we have $\phi(\varphi_0(s))\geq s$ for any $s\in[0,S]$, as $\sigma(\varphi_0(s))\geq s$ for any $s\in]0,S]$ and by \eqref{nul}, \eqref{u_v},  
 we get 
\[
\begin{split}
\phi(\varphi_0(s))
&=\varphi_0(s) + \sum_{l=1}^N \int_{0}^{1} | \w^0_l(s) | ds + \sum_{l=1}^N \int_{\sigma^+(0)}^{ \sigma^+(\varphi_0(s))} \Big\|\frac{d\varphi_l}{ds'}(s')\Big\| ds' \\
&\geq \varphi_0(s) + \sum_{l=1}^N \int_{0}^{\sigma^+(0)} \Big\|\frac{d\varphi_l}{ds'}(s')\Big\| ds' +\sum_{l=1}^N \int_{\sigma^+(0)}^{s} \Big\|\frac{d\varphi_l}{ds'}(s')\Big\| ds' \\
& = \int_0^s \Big[\frac{d\varphi_0}{ds'}(s')+\sum_{l=1}^N \Big\|\frac{d\varphi_l}{ds'}(s')\Big\|\Big] ds' = s,
\end{split}
\]
while $\phi(\varphi_0(0))=\phi(0)=0$ and $\phi(\varphi_0(S))=\phi(h)=S$. From the above relations it also follows that $\phi(\varphi_0(s))=s$ if $\sigma$ is continuous at $\varphi_0(s)$, as in this case $\sigma(\varphi_0(s))=s$.
Suppose by contradiction that there exists $r<\varphi_0(s)$ such that $\phi(r)\geq s$. Since $\phi$ is strictly increasing it follows that $s\leq \phi(r) < \phi(\varphi_0(s))$. Then, in view of the argument above, $\sigma$ is not continuous at $\varphi_0(s)$, hence $\nu$ has a jump and $\phi$ is discontinuous at $\varphi_0(s)$. Therefore, there exist $s_1$, $s_2\in[0,S]$ such that $s\in[s_1,s_2]$ (the largest interval containing $s$ where $\varphi_0$ is constant), $\phi(\varphi_0(s'))=\phi(\varphi_0(s_2))=s_2$ for any $s'\in[s_1,s_2]$ and $\phi(r')<s_1$ for any $r'<\varphi_0(s)$. A contradiction.
Therefore, the functions $\hat\varphi_0$ and $\phi$ coincide with the given  function $\varphi_0$  and its right inverse $\sigma$, respectively.  But then, from statement (i) it readily follows that $(S,\varphi_0,\varphi)$ coincides with the extended control  constructed in part (i) starting from the impulse control $(\mu,\nu,\{\w^r\}_{r\in[0,h]})$. Since $(\tau,y,\beta)$ is the unique extended solution of system \eqref{aux_system} associated with $(S,\varphi_0,\varphi)$, by reading the relations of part (i) in reverse time we see that  $(\mu,\nu,\{\w^r\}_{r\in[0,h]},x,v)$ is an impulse process and also that  the  canonical extended process corresponding to  $(\mu,\nu,\{\w^r\}_{r\in[0,h]},x,v)$ constructed as in statement {\rm(i)} is exactly  $(S,  \varphi_0, \varphi, \tau, y, \beta)$.   \qed

\section{Proof of Theorem \ref{pmp_extended} and Theorem \ref{main:th}} \label{proof_pmpi}
 \begin{proof}[of Thm. \ref{pmp_extended}] Let $(\bar S, \bar\varphi_0,\bar\varphi,\bar\tau,\bar y, \bar\beta)$ be a minimizer for  \eqref{extended_optimal}. Since the extended problem is a conventional control problem, by a standard version of the Maximum Principle  (see e.g.  \cite[Thm. 8.7.1]{OptV}) there exist  $\lambda\geq0$,   adjoint arcs $(q_l^0,q_l)\in W^{1,1}([0,\bar S],\R\times\R^n)$ (associated with $\bar\tau_l$ and $\bar y_l$, respectively) and $d\leq 0$ ($d=0$ if $\bar\beta_N(\bar S)<C$) satisfying relations (B), (C) and \eqref{equal0} of the statement of the theorem together with the non triviality condition
\[
{\rm (A)'} \quad  \|(q_1^0,\dots,q_N^0)\|_{L^\infty}+\|(q_1,\dots,q_N)\|_{L^\infty}+\lambda +|d| \neq 0  \qquad\qquad\qquad\qquad\qquad
\]
and the `constancy of the Hamiltonian' condition
\[
{\rm (D)'} \ \max_{(w_0,w)\in\W} H(\{\bar\tau_l(s)\},\{\bar y_l(s)\}, \{q_l^0(s)\},\{q_l(s)\}, d,w_0,w)=0 \ \ \ \text{for all $s\in[0,\bar S]$.}
\]
Condition (A)$'$  can be strengthened to take the form of conditions (A) in the statement of the theorem. Indeed, first, 
we have
\bel{ntr}
\|(q_1^0,\dots,q_N^0)\|_{L^\infty}+\|(q_1,\dots,q_N)\|_{L^\infty}+\lambda\neq0,
\eeq
as, if it were not true, then $d<0$ and $\bar\beta_N(\bar S)=C>0$. But combining \eqref{equal0} and (D)$'$ and integrating over $[0,\bar S]$ we obtain the contradiction $0=d\bar\beta_N(\bar S)=dC$. Then, supposing that condition (A) is not satisfied,  by the adjoint equation (B) and the transversality condition (C) we should have that $q_1^0=\dots=q_N^0=:q^0$ are constants. Moreover,  condition ${\rm (D)'}$ reads 
\bel{maxHam}
\begin{split}
 Nq^0 \frac{d\bar\varphi_0}{ds}(s) +d \sum_{l=1}^N \Big\| \frac{d\bar\varphi_l}{ds}(s)\Big\| =0 \ \text{ a.e. $s\in[0,\bar S]$}, \quad
 \max_{(w_0,w)\in\W} \{Nq^0 w_0 +d  \| w\|\} = 0.
\end{split}
\eeq
If $d<0$, for $w_0=1$ in   \eqref{maxHam}$_2$ we obtain $q^0\leq 0$, but by  \eqref{maxHam}$_1$ we have $d\bar\varphi_0(s)/ds >0$ a.e., otherwise $ \sum_{l=1}^N \| d\bar\varphi_l(s)/ds \| =1$ on a set of positive measure, which implies $d=0$, a contradiction. Hence, we get
$
q^0 =- \frac{d}{N} \, \frac{ \sum_{l=1}^N \Big\| \frac{d\bar\varphi_l}{ds} (s) \Big\|}{\frac{d\bar\varphi_0}{ds} (s)} \geq 0, 
$
so that $q^0=0$, in  contradiction of \eqref{ntr}. If instead $d=0$, then $q^0\ne0$  and \eqref{maxHam}$_1$ implies that $d\bar\varphi_0(s)/ds =0$ a.e., but then $0=\bar\varphi_0(0)=\bar\varphi_0(\bar S)=h$.
 
Now we prove condition (D). From (D)$'$, it follows that for any $s\in[0,\bar S]$ and each $(w_0,w)\in\W$, we have  
$
H(\{\bar\tau_l(s)\},\{\bar y_l(s)\}, \{q_l^0\},\{q_l\}, d,w_0,w) \leq 0.
$
Therefore, by choosing either $w_0=0$ or $w=0$   we deduce that $H^{\rm (imp)}\leq 0$ and $H^{\rm (dr)}\leq 0$, respectively. 
In particular, \eqref{zero2} follows by choosing $w\in\K^N$ such that $\|w_l\|=1$. 
Now suppose by contradiction that both  $H^{\rm (dr)}$ and  $H^{\rm (imp)}$ are smaller than 0 at some $s\in[0,\bar S]$ and let $(\bar w_0,\bar w)$ be the element in $\W$ that realizes the maximum in (D)$'$. In case either $\bar w_0=0$ or $\bar w=0$ the contradiction is trivial.
In case $\bar w_0>0$ and $\|\bar w\|>0$, then $\hat w:=\bar w/ \sum_{l=1}^N \|\bar w_l\|$ is such that $\sum_{l=1}^N \|\hat w_l\|=1$, hence  it follows
\bel{reasoning}
\begin{split}
&\bar w_0 \Big[\sum_{l=1}^N q_l^0(s) +  \sum_{l=1}^N q_l(s) \cdot f(\bar\tau_l(s),\{\bar y_{l-k}(s)\}) \Big] <0, \\
& \sum_{l=1}^N q_l(s) \cdot \Big( \sum_{j=1}^m g_j(\bar\tau_l(s),\{\bar y_{l-k}(s)\})\hat w_l^j \Big) + d  \sum_{l=1}^N \|\hat w_l\| <0.
\end{split}
\eeq
But multiplying the second relation above by $ \sum_{l=1}^N \|\bar w_l\|$ and adding what is obtained to the first one we get the desired contradiction with (D)$'$.

Let us prove \eqref{zero3}. Let $\NN$ be the null Lebesgue measure subset of $[0,\bar S]$, where \eqref{equal0} is not satisfied. Let $s\in[0,\bar S]\setminus\NN$ be such that $\|d\bar\varphi_l(s)/ds\|>0$ for some $l=1,\dots,N$ and, by \eqref{zero2}, assume by contradiction that 
$
\max_{w\in\K, \, \|w\|=1} q_l(s)\cdot \sum_{j=1}^m g_j(\bar\tau_l(s),\{\bar y_{l-k}(s)\}) w^j+d < 0.
$
Then, reasoning as in \eqref{reasoning} for $\bar w:=d\bar\varphi_l(s)/ds$, we deduce that 
\[
 q_l(s)\cdot \Big(\sum_{j=1}^m g_j(\bar\tau_l(s),\{\bar y_{l-k}(s)\}) \frac{d\bar\varphi^j_l}{ds}(s)\Big)+d \Big\|\frac{d\bar\varphi_l}{ds}(s) \Big\| <0.
\]
Notice that, again by \eqref{zero2} and arguing as in \eqref{reasoning}, the left hand side of the above relation is $\leq0$ when the index $l$ is replaced by any $i\in\{1,\dots,N\}$, so that
%
\[
\sum_{l=1}^N q_l(s)\cdot \Big(\sum_{j=1}^m g_j(\bar\tau_l(s),\{\bar y_{l-k}(s)\}) \frac{d\bar\varphi^j_l}{ds}(s)\Big)+d \sum_{l=1}^N \Big\|\frac{d\bar\varphi_l}{ds}(s) \Big\| <0.
\]
But by \eqref{equal0} and the fact that $H^{\rm (dr)}\leq 0$ and $d\bar\varphi_0(s)/ds \geq0$ we get
\[
\begin{split}
0&= \frac{d\bar\varphi_0}{ds}(s) H^{\rm (dr)}(\{\bar\tau_l(s)\},\{\bar y_l(s)\}, \{q_l^0(s)\},\{q_l(s)\}) \\
&\qquad\qquad+ \sum_{l=1}^N q_l(s)\cdot \Big(\sum_{j=1}^m g_j(\bar\tau_l(s),\{\bar y_{l-k}(s)\}) \frac{d\bar\varphi^j_l}{ds}(s)\Big)+d \sum_{l=1}^N \Big\|\frac{d\bar\varphi_l}{ds}(s) \Big\| <0.
\end{split}
\]
This proves \eqref{zero3}. Now, again by \eqref{zero1} and \eqref{zero2}, we deduce that  
\[
\begin{split}
&\frac{d\bar\varphi_0}{ds}(s) \Big[\sum_{l=1}^N q_l^0(s) +  \sum_{l=1}^N q_l(s) \cdot f(\bar\tau_l(s),\{\bar y_{l-k}(s)\}) \Big] \leq 0\\
&\sum_{l=1}^N q_l(s) \cdot \Big( \sum_{j=1}^m g_j(\bar\tau_l(s),\{\bar y_{l-k}(s)\}) \frac{d\bar\varphi_l^j}{ds}(s) \Big) + d  \sum_{l=1}^N \Big\|\frac{d\bar\varphi_l}{ds}(s) \Big\|\leq0, \quad s\in[0,\bar S]\setminus\NN.
\end{split}
\]
But in view of \eqref{equal0} and \eqref{zero1}, the sum of the two above terms is 0, hence they have to be both 0, so that \eqref{hamdrift} and \eqref{hamimp} are proved.
It remains to show \eqref{hamdrift2} and \eqref{hamimp2}. If $H^{\rm (dr)}<0$ at some $s\in[0,\bar S]\setminus\NN$, then \eqref{hamdrift} implies $d\bar\varphi_0(s)/ds=0$, hence $ \sum_{l=1}^N \|d\bar\varphi_l(s)/ds\|=1$, so that \eqref{hamdrift2} follows by \eqref{zero1}
 and \eqref{hamimp}.
If instead $H^{\rm (imp)}<0$ at some $s\in[0,\bar S]\setminus\NN$, then $d\bar\varphi(s)/ds=0$ (otherwise dividing \eqref{hamimp} by $ \sum_{l=1}^N\|d\bar\varphi_l(s)/ds\|$ we would reach a contradiction). Therefore, $d\bar\varphi_0(s)/ds=1$ and thanks to \eqref{zero1} 
we conclude. \qed
\end{proof}

\begin{proof}[of Thm. \ref{main:th}]
Let $(\bar \mu,  \bar\nu,\{\bar \w^{r}\}_{ r \in [0,h]},\bar x,\bar v)$ be a minimizer of the impulse problem \eqref{impulse_optimal}, to which we associate, for  any $l=1,\dots,N$ and $r\in[0,h]$,   $\bar\zeta_l^r$ and $\bar\theta_l^r$ defined as in \eqref{jump_x1}, \eqref{jump_v1}, respectively. Let $(\bar S,\bar\varphi_0,\bar\varphi,\bar\tau,\bar y,\bar\beta)$ be the corresponding canonical extended process as in Thm. \ref{equivalence_th} and define $(\tilde S, \tilde \tau,\tilde\varphi,\tilde y,\tilde\beta)$ associated with it   as in \eqref{exsol}.   Cor. \ref{invert_cor} yields  
\[
\bar x(t)=\tilde y(\tilde\sigma(t)), \quad \bar\mu([0,t])=\int_0^{\tilde\sigma(t)}\frac{d\tilde\varphi}{ds}(s) ds, \quad \bar\nu([0,t])=\int_0^{\tilde\sigma(t)}\Big\|\frac{d\tilde\varphi}{ds}(s)\Big\| ds, \qquad t\in[0,T],
\]
where $\tilde\sigma$ is the right inverse of $\tilde\tau$. Moreover, for any $l=1,\dots,N$ and $r\in[0,h]$, we have
\bel{wzt}
\begin{cases}
\bar \w_l^r(s)=(\sigma^+(r)-\sigma^-(r)) \frac{d\bar\varphi_l}{ds}(\sigma^-(r)+s(\sigma^+(r)-\sigma^-(r))) \\
\bar\zeta_l^r(s)= \bar y_l(\sigma^-(r)+s(\sigma^+(r)-\sigma^-(r)))\\
\bar\theta_l^r(s)=\bar\beta_l(\sigma^-(r)+s(\sigma^+(r)-\sigma^-(r))), \qquad s\in[0,1],
\end{cases}
\eeq
where $\sigma$ is the right inverse of $\bar\varphi_0$. In particular, for any $r\in[0,h]$ one has $d\bar\varphi_0(s)/ds=0$ a.e. $s\in[\sigma^-(r),\sigma^+(r)]$, so that
\bel{prop_sigma}
\sigma^+(r)-\sigma^-(r) = \int_{\sigma^-(r)}^{\sigma^+(r)} \sum_{l=1}^N \Big\|\frac{d\bar\varphi_l}{ds}(s) \Big\| ds =\int_0^1 \sum_{l=1}^N \|\bar \w_l^r(s)\| ds.
\eeq 
Since $(\bar \mu,  \bar\nu,\{\bar \w^{r}\}_{ r \in [0,h]},\bar x,\bar v)$ is a minimizer for \eqref{impulse_optimal}, then $(\bar S,\bar\varphi_0,\bar\varphi,\bar\tau,\bar y,\bar\beta)$ is a minimizer for \eqref{extended_optimal}, owing to Thm. \ref{equivalence_th}. Hence, appealing to Thm. \ref{pmp_extended}  we deduce the existence of multipliers $(\lambda,d, \{q_l^0\}, \{q_l\})$ satisfying relations (A)--(F). Since $d=0$ if $\bar\beta_M(\bar S)<C$, from \eqref{equiv_tr} it follows  that $d=0$ whenever $\bar v(T)<C$.
Let  $(\tilde q_0,\tilde q):[0,\tilde S]\to{\mathbb R}\times{\mathbb R}^n$ be given by
$
(\tilde q^0,\tilde q)(s):= (q_l^0(s-(l-1)\bar S),q_l(s-(l-1)\bar S))$ for $s\in[(l-1)\bar S,l\bar S]$, $l=1,\dots,N$,
and set 
$$
(p^0 ,p)(t):= (\tilde q^0,\tilde q)(\tilde\sigma(t)), \quad  t\in[0,T],
$$
 and $(p^0(t),p(t))=0$ for $t>T$.  Moreover, for any $l=1,\dots,N$,  $r\in[0,h]$,  and $s\in[0,1]$, set
\bel{def_etalpha}
\eta_l^r(s) = q_l(\sigma^-(r)+s(\sigma^+(r)-\sigma^-(r))), \quad
\alpha_l^r(s) =q_l^0(\sigma^-(r)+s(\sigma^+(r)-\sigma^-(r))).
\eeq
By construction, the transversality condition (iii) is satisfied in view of relation (C) of Thm. \ref{pmp_extended}. Moreover, recalling  the relations between $\tilde\sigma$ and $\sigma$ established in \eqref{2_2_0tilde}, \eqref{2_2_0ih},  
the previous calculations imply $p(T)=\eta_N^h(1)$, $p(0)=\eta_1^0(0)$ and $p((l-1)h)=\eta_l^0(1)$ for any $l=2,\dots,N$. Observe that if $\sigma^+(r)=\sigma^-(r)$, then $d\eta_l^r(s)/ds=0$ and $d\alpha_l^r(s)/ds=0$ for any $l=1,\dots,N$ and a.e. $s\in[0,1]$, so that \eqref{eq_eta} and \eqref{eq_alpha1}  hold true since $\bar \w_l^r\equiv0$. If instead $\sigma^+(r)>\sigma^-(r)$, then $d\bar\varphi_0(s)/ds=0$ a.e. $s\in[\sigma^-(r),\sigma^+(r)]$ and we deduce \eqref{eq_eta} and \eqref{eq_alpha1} by simply deriving the expressions of $\eta_l^r$ and $\alpha_l^r$, using to the adjoint equation (B) in Thm. \ref{pmp_extended} and taking into account \eqref{wzt}. Instead, the boundary conditions in \eqref{eq_eta} and \eqref{eq_alpha1} hold true in view of \eqref{def_etalpha} and thanks to 
condition (C) in Thm. \ref{pmp_extended}. 
Let us prove \eqref{eq_p}. Let $t\in](l-1)h,lh[$ for some $l=1,\dots,N$. From condition (B) of Thm. \ref{pmp_extended}, it follows that
\[
\begin{split}
p(t)&= \tilde q(\tilde\sigma(t)) = \tilde q(\tilde\sigma((l-1)h)) + \int_{\tilde\sigma((l-1)h)}^{\tilde\sigma(t)} \frac{d\tilde q}{ds}(s) ds \\
&=\eta_l^0(1) - \sum_{i=0}^M \int_{\tilde\sigma((l-1)h)}^{\tilde\sigma(t)} \tilde q(s+i\bar S) \cdot \nabla_{x_{i+1}} \tilde f(s+i\bar S) \frac{d\tilde \tau}{ds}(s) ds \\
&\qquad\quad -\sum_{i=0}^M\sum_{j=1}^m \int_{\tilde\sigma((l-1)h)}^{\tilde\sigma(t)}  \tilde q(s+i\bar S) \cdot  \nabla_{x_{i+1}} \tilde g_j(s+i\bar S)\frac{d\tilde\varphi^j}{ds}(s+i\bar S) ds.
\end{split}
\]
Since $\tilde q(\tilde\sigma(t')+i\bar S)=\tilde q(\tilde\sigma(t'+ih))=p(t'+ih)$, using \eqref{change_var} with $A=\tilde\tau$ and $B=\tilde\sigma$ we deduce that
$
\int_{\tilde\sigma((l-1)h)}^{\tilde\sigma(t)} \tilde q(s+i\bar S) \cdot \nabla_{x_{i+1}} \tilde f(s+i\bar S) \frac{d\tilde \tau}{ds}(s) ds$ coincides with  $\int_{(l-1)h}^t  p(t'+ih)\cdot \nabla_{x_{i+1}} \bar f(t'+ih) dt'$
for any $i=0,\dots,M$.
Let $D:=\cup_k [\tilde\sigma^-(r_k), \tilde\sigma^+(r_k)]$, where $\cup_k \{r_k\}$ is the countable set of jumps of $\bar\nu$. Hence, we define $\hat\gamma(s):=0$ for $s\in D$ and $\hat\gamma(s):=d\tilde\varphi(s)/ds$ for $s\in[0,\tilde S]\setminus D$. Accordingly, one has
\[
(\bar\mu^c)^j([0,t])=\int_0^{\tilde\sigma(t)} \hat\gamma^j(s) ds = \int_{[0,t]} \hat\gamma^j(\tilde\sigma(t')) d\tilde\sigma(t'),
\]
where the second equality is a consequence of \eqref{change_var} applied to $A=\tilde\sigma$, $B=\tilde\tau$ and $F(\cdot)=\hat\gamma^j(\tilde\sigma(\cdot))$, and the fact that $\tilde\sigma(\tilde\tau(s))=s$ for all $s\in[0,\tilde S]\setminus D$. As a consequence, it holds
\[
\begin{split}
& \sum_{i=0}^M \sum_{j=1}^m\int_{\tilde\sigma((l-1)h)}^{\tilde\sigma(t)}\tilde q(s+i\bar S) \cdot \nabla_{x_{i+1}} \tilde g_j(s+i\bar S) \frac{d\tilde\varphi^j}{ds}(s+i\bar S) ds \\
&\qquad= \sum_{i=0}^M \sum_{j=1}^m\int_{\tilde\sigma((l-1)h)}^{\tilde\sigma(t)}  \tilde q(s+i\bar S) \cdot \nabla_{x_{i+1}} \tilde g_j(s+i\bar S) \hat\gamma^j(s+i\bar S) ds \\
&\qquad\qquad\qquad\qquad- \sum_{r\in](l-1)h,t]} [\tilde q(\tilde\sigma^+(r) - \tilde q(\tilde\sigma^-(r))] \\
&\qquad =\sum_{i=0}^M \sum_{j=1}^m\int_{(l-1)h}^t  p(t'+ih) \cdot  \nabla_{x_{i+1}} \bar g_j(t'+ih) d(\bar\mu^c)^j(t'+ih) \\
&\qquad\qquad\qquad\qquad -\sum_{r\in]0,t-(l-1)h]} [\eta_l^r(1)-\eta_l^r(0)].
\end{split}
\]
We omit the proof of the adjoint equations \eqref{eq_p01}, as it is completely analogous. 

To prove condition (i),  observe first  that Thm. \ref{pmp_extended},(A)  implies $\lambda+\|\tilde q\|_{L^\infty(0\tilde S)}\neq 0$.
By contradiction, assume that $\lambda+\|p\|_{L^\infty(0,T)}=0$.  In view of the very definition of $p$, it follows that $\tilde q=0$ on $[0,\tilde S]\setminus D$ and, by continuity, also on the boundary of $D$. As a consequence, $\eta_l^r(0)=\eta_l^r(1)=0$ for any $l=1,\dots,N$ and $r\in[0,h]$, so that from the linearity of \eqref{eq_eta} it follows $\eta_l^r\equiv0$. Hence, by \eqref{def_etalpha} it follows $\tilde q=0$ also on $D$, a contradiction.

We deduce condition (iv) using the time change $s=\tilde\sigma(t)$ in the following relation
\[
\max_{w\in\K,\, \|w\|=1} \tilde q(s) \cdot \sum_{j=1}^N \tilde g_j(s) w^j +d\leq0, \quad s\in[0,\tilde S],
\]
which follows from  \eqref{zero2}.

To prove condition (v), take $t\in{\rm supp}(\bar\nu)$ and let $s=\tilde\sigma(t)$. Then, for any $\eps>0$ the set $\{s'\in[0,\tilde S] \text{ : }  \|d\tilde\varphi(s')/ds' \|>0, \ |s-s'|<\eps \}$ has positive Lebesgue measure, so that  \eqref{zero3}, by continuity,  yields (v).

Now, let us fix $r\in]0,h[$ and let $s=\sigma(r)<\bar S$. Then the set $\SS_\eps^s:=\{s'\in[0,\bar S] \text{ : }  d\bar\varphi_0(s')/ds' >0, \ |s-s'|<\eps \}$ has positive Lebesgue measure for any $\eps>0$. In view of \eqref{implica} and by continuity one has $\sum_{l=1}^N [q_l^0(s) + q_l(s) \cdot f(\bar\tau_l(s),\{\bar y_{l-k}(s)\})]=0$. Hence, by \eqref{2_2_0tilde} and \eqref{2_2_0ih} we deduce that
\[
\sum_{l=1}^N [p^0(r+(l-1)h)+p(r+(l-1)h)\cdot \bar f(r+(l-1)h)]=0
\]
for any $r\in]0,h[$. Moreover the previous relation holds at $r=0$ if $\sigma^+(0)=0$ (so that $\SS_\eps^0$ has positive Lebesgue measure for any $\eps>0$), and at $r=h$ if $\sigma^-(h)=\sigma^+(h)$ (so that $\SS_\eps^h$ has positive Lebesgue measure for any $\eps>0$) and $\sigma^+(0)=0$ (so that $\tilde\sigma(lh)=l\bar S$ for any $l=1,\dots,N$). Taking account also of \eqref{prop_sigma}, we have proved condition (vi).

If $\sum_{l=1}^N\int_0^1 \|\bar \w_i^r(s)\| ds >0$ for some $r\in[0,h]$, then $\sigma^+(r)-\sigma^-(r)>0$ and $\sum_{l=1}^N \|d\bar\varphi_l(s)/ds\|=1$ for a.e. $s\in[\sigma^-(r),\sigma^+(r)]$. Hence for a.e.  $s\in[\sigma^-(r),\sigma^+(r)]$ there exists $l=l(s)\in\{1,\dots,N\}$ such that $\|d\bar\varphi_l(s)/ds\|>0$. In view of \eqref{zero1} and \eqref{zero3} and by continuity we deduce that for all $s\in[\sigma^-(r),\sigma^+(r)]$ it holds
\[
\begin{split}
&\sum_{l=1}^N [q_l^0(s) + q_l(s)\cdot f(\bar\tau_l(s), \{\bar y_{l-k}(s)\})]\\
&\qquad\qquad\leq \max_{l=1,\dots,N}\Big[ \max_{\w\in\K,\, \|\w\|=1} q_l(s) \cdot \Big( \sum_{j=1}^m g_j(\bar\tau_l(s), \{\bar y_{l-k}(s)\})\Big) +d \Big] =0.
\end{split}
\]
By \eqref{wzt} and \eqref{def_etalpha} we deduce condition (vii).
Condition (viii) can be proved in a similar way, again  in consequence of \eqref{zero2} and \eqref{zero3}.
\end{proof}

\end{document}